\documentclass[11pt]{paper}
\usepackage{fullpage}

\usepackage {amssymb}
\usepackage {amsmath}
\usepackage {amsthm}
\usepackage [usenames] {color}

\usepackage{tikz}
\usepackage{pgf}
\usepackage{url}
\usetikzlibrary{arrows,automata}
\usetikzlibrary{matrix,arrows,decorations.pathmorphing}
\usepackage[latin1]{inputenc}
\usepackage{tikz,fullpage}
\usepackage{tkz-berge}
\usepackage{verbatim}
\usetikzlibrary{arrows,shapes}
\usepackage{amssymb}
\usepackage{bm}

\usepackage {enumerate}
\usepackage {plaatjes}
\usepackage {cite}
\usepackage{float}
\usepackage [left,pagewise] {lineno}
\usepackage {xspace}
\usepackage {wrapfig}

\definecolor {infocolor} {rgb} {0.6,0.6,0.6}

\definecolor {sepia} {rgb} {0.75,0.30,0.15}
\everymath{\color{sepia}}

\floatstyle{ruled}
\newfloat{algorithm}{thp}{alg}
\floatname{algorithm}{Algorithm}

\newcommand {\Go}{G_{\rm orig}}
\newcommand {\Gs}{G_{\rm rvrs}}

\newcommand {\mathset} [1] {\ensuremath {\mathbb {#1}}}
\newcommand {\R} {\mathset {R}}

\newtheorem {theorem} {Theorem}[section]

\newtheorem {corollary}[theorem] {Corollary}

\newtheorem{lemma}[theorem]{Lemma}
\newtheorem {proposition}[theorem] {Proposition}

\def\Vu{V_{\rm u}}
\def\Vs{V_{\rm s}}

\title{\Large Minimum Input Selection for Structural Controllability}

\author{Alex Olshevsky\thanks{%
   Department of Industrial and Enterprise Systems Engineering,
   University of Illinois at Urbana-Champaign, Urbana, IL, USA, 
   \textsl{aolshev2@illinois.edu}.
  }
}
\date{}
\begin{document}

\maketitle

\begin{abstract} Given a linear system $\dot{x} = Ax$, where $A$ is an $n \times n$ matrix with $m$ nonzero entries,  we consider the problem of finding the smallest set of state variables to affect with an input  so that the resulting system is structurally controllable. We further assume we are given a set of ``forbidden state variables'' $F$ which cannot be affected with an input and which we have to avoid in our selection. Our main result is that this problem can be solved deterministically in $O(n+m \sqrt{n})$ operations. 

%
%
%
\end{abstract}

\section{Introduction}  

 This paper is about the problem of controlling linear systems from a small number of inputs. Our motivation comes from recent interest in the control of systems which are large-scale in the sense of being modeled with a very large number of variables.  
We mention the power grid as one example \cite{Z09,CI12} and biological networks within the human body as another \cite{slotine, R-12}. Since these systems contain a very large number of interacting parts, it does not appear to be sensible to design control strategies for them which are able to affect most (or even many) of these parts. Consequently we study here the general possibility
of  controlling systems from only a few inputs.

Given a linear system,
\begin{equation} \label{withoutinput} \dot{x} = Ax, \end{equation} where $A \in \R^{n \times n}$ is given, we consider whether it is possible to choose an input matrix $B \in \R^{n \times n}$ such that the 
resulting system with input, 
\begin{equation} \label{withinput} \dot{x} = Ax + Bu, \end{equation} is controllable, and the matrix $B$ has the smallest possible 
number of rows with a nonzero entry. Note that each row of $B$ with a nonzero entry corresponds to a variable of the system of Eq. (\ref{withoutinput}) affected with an input. In addition, we assume that we are also given a set $F \subset \{1, \ldots, n\}$ consisting of variables which cannot be affected with an input; this means that the corresponding rows of $B$ have to consist entirely of zero entries. Intuitively, some variables of the system may be out of reach of any actuator or control strategy we design and should not be considered as possible input locations. 


Unfortunately, it was recently observed  that this problem is NP-hard even in the case when $F = \emptyset$ \cite{previous}. We are therefore forced to consider possible relaxations of the problem which may be solvable in polynomial time. One approach, studied in \cite{previous}, is to study approximation algorithms, i.e., to try to find a $B$ in polynomial time which does not have too many nonzero rows compared to the optimal matrix.

In this paper we consider a different relaxation of the problem, pioneered by the recent papers \cite{slotine} and \cite{pequito-main, pequito-main-acc, dion-main}: we would like Eq. (\ref{withinput}) to be structurally controllable rather than controllable. A formal definition of structural controllability is given in Section \ref{background}, but loosely speaking this means we are considering controllability for arbitrarily small perturbations of the nonzero entries of the matrix $A$. In many cases, the nonzero entries of the matrix $A$ are not precisely known and little is lost by instead considering system controllability after an arbitrary small perturbation of them. We will refer to this (i.e., to the problem of finding $B$ with fewest nonzero rows making Eq. (\ref{withinput}) structurally controllable) as the minimum structural controllability problem.


\subsection{Previous work and our results in this paper} 

Similar questions have recently been considered in the recent works of \cite{slotine} and \cite{pequito-main, pequito-main-acc, dion-main} in the setting when $F= \emptyset$, i.e., there are no forbidden variables. In \cite{slotine}, the question of finding a $B$ with the smallest number of columns with a nonzero entry was considered. The number of columns of $B$ with a nonzero enry corresponds to the number of components of the vector $u$ in Eq. (\ref{withinput}) which end up affecting the system; intuitively, this is a measure of the number of independent signals which are needed to control the network. In the language of \cite{slotine}, each such entry of $u$ corresponds to a ``driver node.'' It was shown in \cite{slotine} that if $A$ has with $m$ nonzero entries, then a matrix $B$ rendering Eq. (\ref{withinput}) structurally controllable with fewest number of columns with a nonzero entry can be found in 
$O(n+m \sqrt{n})$ operations. 

There is a close connection between structural controllability and the problem of finding maximum matchings in a graph, as pointed out by \cite{murota, kyber, slotine} and as we reprise in Section \ref{sec:combinatorial} in this paper. The significance of the  running time $O(n + m \sqrt{n})$ achieved by \cite{slotine} is that it is  the same as the best currently known deterministic complexity for finding a maximum cardinality matching in a bipartite graph\footnote{To be more precise, when a graph is given as an ``adjacency list,'' i.e., a table with the $i$'th row containing an enumeration of the neighbors of $i$, the best known deterministic complexity of finding a maximum matching is 
$O(n+m \sqrt{n})$ using the Hopcroft-Karp algorithm \cite{hk}.}.

Although at first glance the problem of finding a $B$ with fewest number of columns with a nonzero entry appears extremely similar to the minimum structural controllability problem, the two problems are quite different. Intuitively, there is no reason why the number of independent signals needed to control a network should have a close relationship with the number of variables of the system which needs to be affected. 

To illustrate this,  consider the
case when $A$ is diagonal with every diagonal entry nonzero. It is immediate that the system is structurally controllable with a $B$ with only a single
column with nonzero entries, i.e., we may take the first column of $B$ to be the all-ones vector and set the remaining entries of $B$ to zero\footnote{One way to prove quickly that this system is structurally controllable is to argue that the nonzero diagonal entries of $A$ can be 
perturbed to be distinct and then the controllability matrix contains an invertible Vandermonde matrix as a submatrix. Alternatively, this also follows immediately from \cite{slotine} or the classic results of \cite{lin}.}. On the other hand, the smallest number of rows with  nonzero entries in a $B$ making such a system structurally controllable is $n$: since there is no coupling between variables, it is immediate that every variable needs to be affected. 

To our knowledge, the first papers considering the minimum structural controllability problem were the recent works \cite{dion-main, dion1, dion2, dionrecent} and \cite{pequito-main, pequito-main-acc}. In \cite{dion-main, dion1, dion2} graph theoretic conditions and bounds were given for several variations of the minimal structural controllability problem, ultimately bounding the number of additional inputs needed in terms of critical connection components and rank defects in corresponding graphs. An explicit characterization of the solution of the minimal structural controllability problem with a single input was given in \cite{dionrecent}. In \cite{pequito-main-acc}, an algorithm for the minimal structural controllability problem was proposed. Although formally this paper considered the problem of finding a matrix $B$ with fewest nonzero entries in total, it is not hard to see this is equivalent to the minimum structural controllability problem (see Section \ref{background} for a discussion of this). This algorithm took $O(m n^{1.5})$ operations to produce a solution. The paper \cite{pequito-main} then provided an $O(n^3)$ algorithm for this problem. 

Our contribution in this paper is to provide an algorithm for the minimum structural controllability problem (additionally with an arbitrary set of forbidden variables $F \subset \{1, \ldots, n\}$) which runs faster,
namely in $O(n + m\sqrt{n})$ operations. This is always faster by a factor of $\sqrt{n}$ compared to the previously best running times in \cite{pequito-main, pequito-main-acc} and is better by a factor of $n$ when the graph corresponding to $A$ is sparse, i.e., if $m=O(n)$. More importantly, our finding is that it is possible to solve the minimum structural controllability problem as fast as the currently-best complexity of deterministically finding a maximum matching in a bipartite graph.

\subsection{Literature overview} The concept of structural controllability was introduced in the groundbreaking paper of Lin \cite{lin}, which provided a combinatorial necessary and sufficient condition for a system with given matrices $A,B$ to be structurally controllable. Lin's work was elaborated upon in a number of now-classic works in the 1970s and 1980s. We mention specifically \cite{sp,structured-followup-76a, structured-followup-76b, structured-followup-77, structured-followup-79a, structured-followup-79b, structured-followup-81, structured-followup-82, structured-followup-84, structured-followup-84b, structured-followup-86} which refined Lin's work in a number of ways. 

Shortly after Lin's paper,  Shields and Pearson \cite{sp} generalized Lin's result to the case when $B$ is a matrix (Lin had only studied the case when $B$ belongs to $ \R^{n \times 1}$); alternative, shorter, proofs of Lin's main results were provided by Glover and Silverman \cite{structured-followup-76b} as well as by Hosoe and Matsumoto \cite{structured-followup-79a}; Corfmat and Morse considered the case when the $A$ and $B$ were parametrized \cite{structured-followup-76a};   stronger notions of structural controllability was  proposed by Mayeda and Yamada \cite{structured-followup-79b}  and Willems \cite{structured-followup-86}; and the related notion of structurally fixed modes was studied by Sezer and Siljak \cite{structured-followup-81, structured-followup-84} as well as Papadimitriou and Tsitsiklis \cite{structured-followup-84b}. We are not able to survey the entire classic literature on the subject and instead point the reader to the relatively recent survey \cite{structured-survey}.  

There has been  considerable contemporary interest in structural controllability as well as minimum controllability problems as a result of the recent Nature paper of Liu, Slotine, and Barabasi\cite{slotine}. We mention \cite{LSB13} by the same authors which studied the applications of this framework to the observability of biological networks, as well as \cite{jia, LSB12, pljb12} by the same research group which examined the effects of network statistics on controllability. We have already described the recent works of Commault and Dion \cite{dion-main, dionrecent} and Pequito, Kar, and Aguiar \cite{pequito-main, pequito-main-acc} which are the most closely related papers to this work. The earliest reference on such problems we are aware of is the work of Simon and Mitter from 1960's \cite{simon-mitter} which considers synthesizing observers which take as few as possible measurements of the state. We also mention \cite{npc} which studies whether (non-structural) minimum controllability problems are NP-complete as well as \cite{mc} which studies application of controllability problems to model checking. Finally, structural controllability over finite fields was investigated by Sundaram and Hadjicostis \cite{shreyas}.

There has also been much interest in input selection for strong structural controllability problems (introduced in the 1970s by Mayeda and Yamada \cite{structured-followup-79b}) wherein the requirements to be satisfied are more stringent, namely that the system has to be controllable for arbitrary perturbations to its nonzero entries. Unfortunately, it turns out that in the setting of strong structural controllability, input selection problems tend to be NP-hard; two recent references establishing such results are \cite{chapman} and \cite{nima} (although \cite{pequito-dedicated} demonstrates that some variations of minimal strong structural controllability problems are nevertheless polynomial time).

A closely related strand of work studies input selection for minimum-energy control; we refer the reader to \cite{bullo, Motter1, Motter2, lygeros, lygeros2}. For multi-agent systems with nearest-neighbor interactions, controllability was investigated nearly a decade ago by Tanner \cite{t04} and Ji, Muhammed, and Egerstedt \cite{Eg3} with recent work in \cite{cao-recent, notar1, notar2, chapman2, nabi, Eg4}. 

We remark that minimal controllability problems such as the one we consider here are closely related to the recent literature on network controllability which seeks to relate
graph-theoretic properties of network to controllability. The development of easily optimizable necessary and sufficient conditions for controllability properties of networks would have
immediate consequences for the input selection problems of the kind we consider. However, in the non-structural case, such conditions appear to be challenging to obtain, though 
much can be said in some particular cases. We refer the reader to \cite{Eg2, nabi, notar1, notar2} as well as the recent survey \cite{Eg1} which provides an overview of the area. 

\section{Problem statement and our result\label{background}}

We now give a formal statement of the problem we will be considering as well as of our main result. 
We begin with a brief introduction to the notion of structural controllability.

\bigskip

\medskip

We define the zero pattern of a matrix $P$, denoted by $Z(P)$, to be the set of entries $(i,j)$ such that $P_{ij}=0$. 
Given two matrices $A, B$  the linear system of Eq. (\ref{withinput})  is called structurally controllable if there exist matrices $A',B'$ with the same dimensions as $A,B$, which satisfy \begin{eqnarray*} Z(A) & \subset&  Z(A') \\
Z(B) & \subset & Z(B') 
\end{eqnarray*}   such that the linear system \begin{equation} \label{perturb} \dot{x} = A' x + B' u \end{equation} is controllable.
The concept of structrual controllability was introduced in the pioneering work of Lin \cite{lin}, and it was shown in \cite{lin,sp} that if Eq. (\ref{withinput}) is structurally controllable, then in fact the linear system of Eq. (\ref{perturb}) is controllable for allmost all pairs of matrices $A',B'$ whose zero sets contain the zero sets of $A$ and $B$. In particular, if Eq. (\ref{withinput}) is structurally controllable, then it is possible to perturb the nonzero entries of $A$ and $B$ by an arbitrarily small amount and obtain a controllable system. 

\bigskip

\medskip

Here we will be concerned with what we call the ``minimum structural controllability problem,'' which we describe now (actually, we describe a particular version of the problem which we will see is equivalent to the general case). Given a matrix $A \in \R^{n \times n}$ and set of forbidden variables $F \subset \{1, \ldots, n\}$ we seek to find a set $I \subset \{1, \ldots, n\}$ of minimum cardinality such that $I \cap F = \emptyset$ and
\begin{equation} \label{sys-input} \dot{x}=Ax+B(I) u\end{equation} is structurally controllable, where $B(I)$ is some diagonal matrix  satisfying \begin{equation} \label{bcond} B_{ii}(I) \neq 0 \mbox{ if and only if } i \in I. \end{equation}  Observe that each nonzero diagonal entry of $B(I)$ corresponds to a variable of the linear system $\dot{x}=Ax$ affected with an input, while each zero diagonal entry corresponds to a variable unaffected.
%

Note that the minimum structural controllability problem may not have a solution, for example if $F = \{1, \ldots, n\}$.  When a solution does exist, we will adopt the convention of saying the minimum structural controllability problem is solvable. Furthermore, note that the actual nonzero diagonal values of $B(I)$ do not matter, i.e.,  if Eq. (\ref{sys-input}) is structurally controllable with one diagonal $B(I)$
satisfying Eq. (\ref{bcond}) then it is structurally controllable with all such $B(I)$\footnote{Indeed, since the definition of structural controllability is based on arbitrary perturbation of the nonzero entries of $A,B$, the actual values of those entries never matter.}. 

We remark that this is equivalent to the problem of finding $B$ having the fewest number of rows with a nonzero entry such that Eq. (\ref{withinput}) is structurally controllable. Indeed, given any $B$ making Eq. (\ref{withinput}) structurally controllable, we can simply set $I$ to be the set of rows of $B$ with a nonzero entry, and then any matrix $B(I)$ 
satisfying Eq. (\ref{bcond}) renders Eq. (\ref{withinput}) structurally controllable. Thus nothing is lost by searching for diagonal matrices $B$. 

Furthermore, the minimum structural controllability problem is also equivalent to the problem of finding $B$ having fewest nonzero entries making Eq. (\ref{withinput}) structurally controllable. The reasoning is the same as in the previous paragraph: given $B$, define $I$ once again to be the set
of rows of $B$ with a nonzero entry, and we then we have that $B(I)$ cannot have
more nonzero entries than $B$.

\bigskip

\medskip

This paper analyzes the complexity of solving the minimum structural controllability problem in terms of the problem parameters, which are $n$ and $m$ (recall these are, respectively, the dimension of $A$ and the number of nonzero entries in $A$). We assume that $A$ is given to us in the form of a list of of all the entries $(i,j)$ such that $A_{ij} \neq 0$; and the set $F$ of forbidden nodes is given to us as a list of entries in $\{1, \ldots, n\}$. We will use the standard unit-cost RAM model of computation. For convenience, we define 
the graph $G(A)$ to be the directed graph with the vertex  
set $\{1, \ldots,n\}$  and edge set $E(A) = \{(i,j) ~|~ A_{ji} \neq 0\}$. We will refer to $G(A)$ as the adjacency graph of the matrix $A$.

As previously mentioned, our main result is an algorithm which finds a set $I$ asked for by the minimum structural controllability problem, or declares that no such set exists, in $O(n + m \sqrt{n})$ operations. This will be done under an assumption which we now describe and which carries no loss of generality.
 {\em We will make the assumption in the remainder of the paper that no node in $G(A)$ is isolated (a node is isolated if it has no incoming or outgoing edges, which means the corresponding row and column of $A$ is identically zero).} This can indeed be done without loss of generality since all variables corresponding to isolated nodes clearly need to be affected by inputs, and since all isolated nodes can be enumerated straightforwardly in $O(n+m)$ operations.

We now describe the structure of the remainder of the paper. In the following Section \ref{sec:combinatorial} we describe a combinatorial reformulation
of the minimal structural controllability problem. It is this combinatorial reformulation which we then proceed to solve in the following Section \ref{sec:main}. 
Note that our final result, namely a running time of $O(n + m \sqrt{n})$ operations, is a combination of a series of reductions made throughout the paper. More precisely, 
this running time follows from putting together Proposition \ref{reformulation} on the combinatorial reformulation of the problem, Proposition \ref{prop:allexists} and the discussion at the end of Section \ref{sec:hk} on the running time of the Hopcroft-Karp algorithm, and Proposition \ref{prop:opcount} \& Theorem \ref{thm} which bound the running time of our main algorithm in Section \ref{sec:main}.

\section{A combinatorial reformulation\label{sec:combinatorial}} 

Since the values of the nonzero entries of the matrices $A,B$ do not appear in the definition of structural controllability, it is usually convenient to restate questions about structural controllability in terms of graphs corresponding to these matrices. Here we describe such a combinatorial reformulation of the 
minimum structural controllability problem which will be the basis for the remainder of this paper.  We do not claim any novelty for this reformulation as it is a trivial modification of Theorem 8 from \cite{pequito-main}   and Theorem 10 from \cite{dion-main}. 

\bigskip

 We define a {\sl partitioned directed graph} to be an ordinary directed graph $G=(V,E)$ equipped with a partition of the 
set of vertices $V = \Vu \cup \Vs, ~~ \Vu \cap \Vs = \emptyset$ such that the edge set $E$ contains no edges whose destination is in $\Vu$. Given the linear system of Eq. (\ref{withinput}) where  $A$ has dimensions $n \times n$ while $B$ has dimensions $n \times k$, we will associate a partitioned directed graph by setting $$\Vs = \{1, \ldots, n\}, ~~\Vu = \{1', \ldots, k'\}$$ and defining the edge set $E$ to consist of all the edges $(i,j)$ with $A_{ji} \neq 0$ and $(i',j)$ with $B_{ji'} \neq 0$. Some examples are drawn in Figures \ref{ex1} and \ref{ex2}. 

\newsavebox{\smlmat}
\savebox{\smlmat}{$\dot{x} = \left(\begin{smallmatrix}0 &1 \\0 &1\end{smallmatrix}\right)x+\left( \begin{smallmatrix}1\\0\end{smallmatrix} \right) u$}

\newsavebox{\aaa}
\savebox{\aaa}{$\dot{x} = \left(\begin{smallmatrix}0 &0 \\0 &1\end{smallmatrix}\right)x+\left( \begin{smallmatrix}2\\3\end{smallmatrix} \right) u$}

\newsavebox{\bb}
\savebox{\bb}{$\dot{x} =  \left( \begin{array}{cc} 1 & 2  \\ 1 & 2 \\ 1 & 2\end{array} \right) u$}

\begin{figure} \begin{center}

\begin{tikzpicture}[->, thick]
\SetVertexNormal[LineColor=black]
\SetVertexMath

\node (A1) at (-10,0) [circle, draw] {$1$};
\node (A2) at (-7.5,0) [circle, draw] {$1'$};

\node (B1) at (-4,0) [circle, draw] {$2$};
\node (B2) at (-1.5,0) [circle, draw] {$1$};
\node (B3) at (1,0) [circle, draw] {$1'$};

\path[->, color=red]

(B3) edge (B2)
(B2) edge (B1)
(B1) edge [out = 115, in = 55, looseness = 4] (B1);

\path[->,color=red]

(A2) edge (A1);

\end{tikzpicture} 

\caption{On the left is a drawing of the partitioned graph corresponding to the scalar equation $\dot{x} =  u$; on the right is a drawing of the partitioned graph corresponding to  \usebox{\smlmat} .}
\label{ex1}
\end{center} \end{figure}
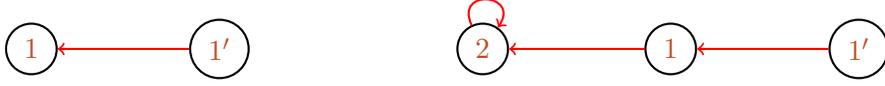

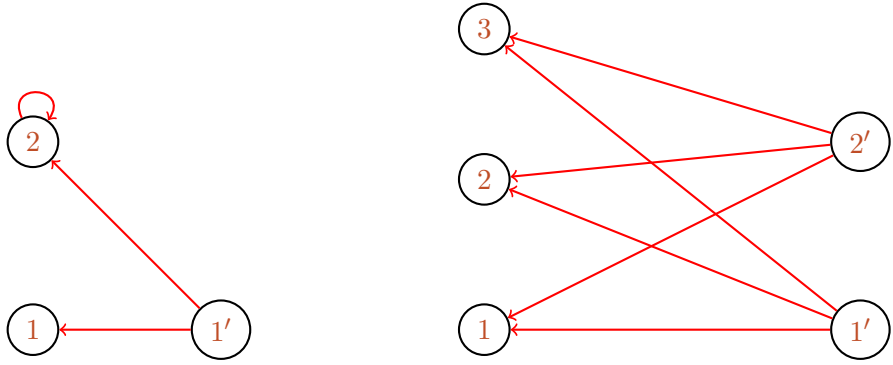
\begin{figure} \begin{center}

\begin{tikzpicture}[->, thick]
\SetVertexNormal[LineColor=black]
\SetVertexMath

\node (A1) at (-10,0) [circle, draw] {$1$};
\node (A2) at (-7.5,0) [circle, draw] {$1'$};
\node (A3) at (-10,2.5) [circle, draw] {$2$};

\node (U1) at (1,0) [circle, draw] {$1'$};
\node (U2) at (1,2.5) [circle, draw] {$2'$};

\node (V1) at (-4,0) [circle, draw] {$1$};
\node (V2) at (-4,2) [circle, draw] {$2$};
\node (V3) at (-4,4) [circle, draw] {$3$};

\path[->, color=red]

(U1) edge (V1)
(U1) edge (V2)
(U1) edge (V3)
(U2) edge (V1)
(U2) edge (V2)
(U2) edge (V3);

\path[->,color=red]

(A2) edge (A1)
(A2) edge (A3)
(A3) edge [out=115, in = 55, looseness=4] (A3);

\end{tikzpicture} 

\caption{On the left is a drawing of the partitioned graph corresponding to the equation \usebox{\aaa} on the right-hand side is a drawing of the partitioned graph corresponding to the equation  \usebox{\bb} .}
\label{ex2}
\end{center} \end{figure}

 For any directed graph (ordinary or partitioned), we will use the following notation: given a subset of the vertices $S$, we will use $N_{-}(S)$ to refer to 
the set of in-neighbors of $S$ 
and $N_+(S)$ will refer to the set of out-neighbors of $S$. A subset of the vertices $S$ is called contracting
if $|N_{-}(S)|<|S|$. 

\begin{theorem}[\cite{lin,sp}] \label{sp-theorem} The linear system of Eq. (\ref{withinput}) is structurally controllable if and only if the associated partitioned 
graph $G= ( \Vu \cup \Vs, E)$ satisfies the following two conditions:
\begin{enumerate} \item Any node in $\Vs$ is reachable by a path starting from some node in $\Vu$. 
\item No subset of $\Vs$ is contracting. 
\end{enumerate}
\end{theorem}

In a slight abuse of notation, we will now say that a partitioned graph is structurally controllable if it satisfies both of these conditions. 
For example, both graphs in Figure \ref{ex1} are structurally controllable, as is the graph on the left in Figure \ref{ex2}; however, the 
graph on the right in Figure \ref{ex2} is not structurally controllable. 

We define a matching
$M$ in a directed graph to be a subset of the edges such that no two edges in $M$ have a common source or a common destination. We will say that
a vertex $v$ is unmatched with respect to a matching $M$ if there is no edge in $M$ which has $v$ as its destination. We will use $U(M)$ will denote the set of unmatched nodes in the matching $M$. A matching $M$ in a directed graph is said to be perfect $U(M) = \emptyset$, i.e., if no node is unmatched.
We extend this definition to partitioned graphs as follows. Note that in a partitioned graph $G = (\Vu \cup \Vs, E)$, no node of $\Vu$ has any incoming edges;  consequently, we will say that a matching $M$ in a 
partitioned graph $G$ is perfect if $U(M) = \Vu$,  i.e., if no node in $\Vs$ is unmatched. 

With these definitions in place, it was observed in \cite{slotine} observed that condition (2) of Theorem \ref{sp-theorem} may be restated more conveniently in 
terms of matchings.

\begin{theorem}[\cite{slotine}] \label{slotine-theorem} Condition (2) of Theorem \ref{sp-theorem} holds if and only if there exists a 
perfect matching in the partitioned graph $G= ( \Vu \cup \Vs, E)$.
\end{theorem}

\bigskip

As a consequence of Theorem \ref{sp-theorem} and Theorem \ref{slotine-theorem}, we can reformulate the minimum structural controllability problem
in combinatorial terms. 
This will require several more definitions. 

Given a matrix $A \in \R^{n \times n}$, we can define 
the graph $G(A)$ to be the directed graph with the vertex  
set $\{1, \ldots,n\}$  and edge set $E(A) = \{(i,j) ~|~ A_{ji} \neq 0\}$. We will refer to $G(A)$ as the adjacency graph of $A$. 
Given a directed graph $G$ and the set of forbidden vertices $F$, a matching $M$ in $G$ is called an allowed matching if $U(M) \cap F = \emptyset$, i.e., if no node in $F$ is unmatched. F
We define the condensation of any
graph to be the directed acyclic graph obtained by collapsing together the strongly connected components; the condensation of $G(A)$ will be denoted by
$G_{\rm cond}(A)$.  In any directed graph, we will say that a vertex is a source vertex if it has no incoming edges. We will say that a connected component of $G(A)$ is a source connected component if it collapses to a source vertex in $G_{\rm cond}(A)$. Finally, given a matching $M$ in a directed graph $G$, we define the cost of the matching to be the number of unmatched vertices plus the number of source strongly connected components without an unmatched node.

We can now state the combinatorial reformulation of the minimum structural controllability problem. We mention once again that we claim no novelty for this reformulation as it is an immediate  consequence of  Theorem  8 from \cite{pequito-main} as well as Theorem 10 from \cite{dion-main}.

\smallskip

\begin{proposition}  \label{reformulation} The minimum structural controllability problem is solvable if and only if \begin{enumerate} \item  $G(A)$ has an allowed matching. \item Every source connected component of $G(A)$ has a node in $F^c$.  \end{enumerate}
Furthermore, the optimal set $I$ asked by the minimum structural controllability problem can be recovered in $O(m)$ operations from the minimum cost allowed matching in $G(A)$. 
\end{proposition}

\begin{proof} We first observe that due to item (1) of Theorem \ref{sp-theorem},  if some source strongly connected component of 
$G(A)$ does not have a node in $F^c$, then the minimum structural controllability problem is not solvable. Since a listing of the source strongly connected components 
of a directed graph may be obtained in $O(m)$ time using, for example, Kosaraju's algorithm \cite{aho}\footnote{The running time of Kosaraju's algorithm is frequently cited as $O(m+n)$ operations, but we have assumed that no node in $G(A)$ is isolated so that $n \leq 2m$.} and further checking that any source connected component has  an entry in $F^c$ can be done in $O(n)$ additional time, the conclusion that ``the minimum structural controllability problem is not solvable'' can always be read off in $O(m)$ operations in this case. 

We therefore only need to consider the case when each source strongly connected component of $A$ has a  vertex in $F^c$, {\em which we do for the remainder of this proof}. Let $I^*$ be the smallest possible cardinality of the sought-after set $I$ in the minimum structural controllability problem and let $c^*$ be
the smallest cost of any allowed matching in $G(A)$. If the minimum structural controllability problem is not solvable, we say $I^* = +\infty$; and if no allowed matching exists, we will say $c^* = +\infty$. We next argue that $I^* = c^*$. 

We first  argue that $c^* \leq I^*$. Indeed, let $I$ be a set of indices such that $I \cap F = \emptyset$ and such that Eq. (\ref{withinput}) is structurally controllable with some diagonal matrix $B(I)$ satisfying $B_{ii}(I) \neq 0$ if and only if $i \in I$. Consider the associated partitioned directed graph.
By item (2) of Theorem \ref{sp-theorem} and Theorem \ref{slotine-theorem} we have that there is a perfect matching in this partitioned graph. By considering the edges of this matching which have both source and destination in $\Vs$, we obtain a matching in $G(A)$. Call this matching $M$. 

Since after we add to $M$ some edges going from $\Vu$ to $\Vs$ we get a perfect matching, it follows that every unmatched vertex in $M$ has an incoming edge from $\Vu$, and consequently belongs to $F^c$. Thus $M$ is an allowed matching. Moreover, by item (1) of Theorem \ref{sp-theorem}, we have that every source connected component of $G(A)$ without an unmatched node in $M$ has an incoming link from $\Vu$. Since $B(I)$ is a diagonal matrix,  we have that the cardinality of $I$ exactly equals the number of nodes in $\Vs$ with an incoming link from $\Vu$. It follows that the cost of $M$ is at most the cardinality of $I$. This proves that $c^* \leq I^*$. 

We next argue that $I^* \leq c^*$. Indeed, suppose $M$ is an allowed matching in $G(A)$. We can construct  a set $I_M$ by first putting into it every variable corresponding to a vertex in $U(M)$ and then adding an arbitrary node in $F^c$ from each strongly connected component without an unmatched node. Since $M$ is an allowed matching, we have $I_M \cap F = \emptyset$. Now take any diagonal matrix $B(I_M)$. In the associated partitioned directed graph, we have that item (1) of Theorem \ref{sp-theorem} is satisfied since we have ensured there is an outgoing link from $\Vu$ to each source strongly connected component of $G(A)$. We finally argue that item (2) of Theorem \ref{sp-theorem} is satisfied as well. Indeed, by construction, each node in $U(M)$ has an incoming ede from $\Vu$; therefore we can construct a perfect matching in the partitioned graph by taking together the edges in $M$ with all the edges from $\Vu$ incoming on nodes in $U(M)$. Appealing to Theorem \ref{slotine-theorem}, we conclude that item (2) of Theorem \ref{sp-theorem} holds, and consequently the set $I_M$ makes Eq. (\ref{withinput}) structurally controllable. Observing that the cardinality of $I_M$ is exactly the cost of $M$, 
we conclude the proof that $I^* \leq c^*$. 

We have thus shown that $c^* = I^*$. In particular, under the condition that every source connected component of $G(A)$ has a node in $F^c$, we have
that $c^*$ is finite if and only if $I^*$ is finite. This proves the first equivalence of this proposition. Finally, to argue that the optimal set $I^*$ can be read off the minimum cost allowed matching $M^*$ in $O(m)$ operations, observe that the previous paragraph has described exactly how to construct  $I^*$ from the minimum cost matching $M^*$: all we need to do is  list all the unmatched 
vertices in $M^*$ as well as all the vertices lying in the intersections between each source strongly connected components without an unmatched node in $M^*$ and $F^c$. This can  be straightforwardly done in $O(m)$ operations once a listing
of the strongly connected component is available, which as we already remarked takes $O(m)$ operations to compute using Kosaraju's algorithm. Thus in the end it takes $O(m)$ operations
to find the optimal set $I^*$ once the minimum cost allowed matching has been found. 
\end{proof}

\section{Finding a minimum cost allowed matching\label{sec:main}}

Having proved Proposition \ref{reformulation}, we need only concern ourselves with a purely combinatorial question: given a directed graph $G$ and a set of forbidden vertices $F$, how do we find an allowed matching of minimum cost (or declare that no allowed matching exists)? In this section, we describe how to solve this problem in $O(m \sqrt{n})$ operations. Coupled with Proposition \ref{reformulation}, this immediately implies our main result, namely that the minimum structural controllability problem is solvable in $O(n+m \sqrt{n})$ operations (note that the additional factor of $n$ comes from making sure the graph $G(A)$ has no isolated nodes; we assumed this in the previous section, but ensuring this requires $O(m+n)$ operations as we previously noted). 

\bigskip

\medskip

Our algorithm has two parts. First, we will describe how to find an allowed matching (or declare none exists) in $O(m \sqrt{n})$ operations.
This is done in Section \ref{sec:hk} and is basically an immediate application of the well-known Hopcroft-Karp algorithm for maximum bipartite matchings. Next, we describe 
an augmentation process which, starting from an allowed matching, produces a minimum cost allowed matching in $O(m\sqrt{n})$ additional operations. 
This is described in Section \ref{sec:augment}. Putting together the results of these two sections immediately gives that the minimum cost allowed matching problem can be solved
in $O(m \sqrt{n})$ operations.

\subsection{Finding an allowed matching, if it exists\label{sec:hk}}

In this section we address the question of finding an allowed matching in a graph. It could very well be that, for some directed graph $G$ and set of forbidden vertices $F$, no allowed matching exists; for example, if two nodes in $F$ have have in-degree one with the same in-neighbor, at least one of them is bound to be unmatched.

Here we reformulate this problem as a a bipartite matching problem which can then be solved using the Hopcroft-Karp algorithm in $O(m \sqrt{n})$ operations. Along the way, we introduce some definitions which will be useful to us.

\bigskip

\medskip

Given a directed graph $G$, the splitting of $G$ is defined to be the directed bipartite graph obtained as follows: for every node $u$, we create two nodes $u_{\rm src}$ and $u_{\rm dst}$, and for every edge $(u,v)$ in the original graph we put the edge $(u_{\rm src}, v_{\rm dst})$. We will refer to all the nodes $u_{\rm src}$ as ``source nodes'' and to the nodes $v_{\rm dst}$ as ``destination nodes.'' We will say the edge $(u,v)$ in $G$ and the edge $(u_{\rm src}, v_{\rm dst})$  ``correspond to each other.'' See Figure \ref{fig-split} for an example of a graph and its splitting.

\begin{figure}
\begin{tikzpicture}[->, thick]
\SetVertexNormal[LineColor=black]
\SetVertexMath

\node (d) at (-12.5,0) [circle, draw] {$d$};
\node (c) at (-10,0) [circle, draw] {$c$};
\node (b) at (-7.5,0) [circle, draw] {$b$};
\node (a) at (-5,0) [circle, draw] {$a$};

\node (dd) at (3,-2) [circle, draw] {$d_{\rm dst}$};
\node (cd) at (3,0) [circle, draw] {$c_{\rm dst}$};
\node (bd) at (3,2) [circle, draw] {$b_{\rm dst}$};
\node (ad) at (3,4) [circle, draw] {$a_{\rm dst}$};

\node (ds) at (-2,-2) [circle, draw] {$d_{\rm src}$};
\node (cs) at (-2,0) [circle, draw] {$c_{\rm src}$};
\node (bs) at (-2,2) [circle, draw] {$b_{\rm src}$};
\node (as) at (-2,4) [circle, draw] {$a_{\rm src}$};

\path[->, color=red]

(a) edge [bend left] (b)
(b) edge (c)
(c) edge (d)
(a) edge [out=115, in = 55, looseness=4] (a)
(b) edge [bend left] (a)
(b) edge [out=115, in = 55, looseness=4] (b)
(as) edge (ad)
(as) edge (bd)
(bs) edge (bd)
(bs) edge (ad)
(bs) edge (cd)
(cs) edge (dd);

\end{tikzpicture} \caption{The graph on the right is the splitting of the graph on the left.}\label{fig-split}
\end{figure}
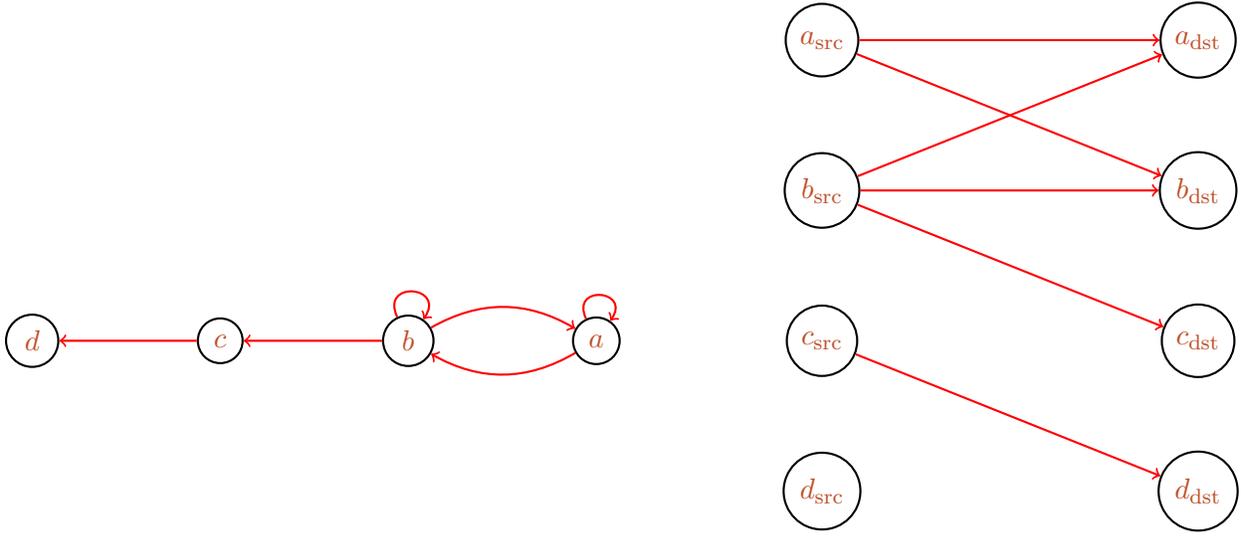

Given a directed graph $G=(V,E)$ and a set $V' \subset V$,we will say that the {\em subgraph determined by $V'$} is
the graph with vertex set $V' \cup N_{\rm in}(V') \cup N_{\rm out}(V')$ and edge set $(a,b) \in E$ such that at least one of $a,b$ belongs to $V'$. 

We then have the following fact.

\begin{proposition} Let $G_F$ be the graph obtained by taking the subgraph of the splitting of $G$ determined by the destination vertices of $F$ and  viewing it as an undirected graph by ignoring the orientations of the edges. Then an allowed matching in $G$ exists if and only if the maximum cardinality matching 
in $G_F$ is of size $|F|$.  \label{prop:allexists}

Moreover, an allowed matching can be recovered in $O(n)$ operations from a matching in $G_F$ of size $|F|$.
\end{proposition}

\begin{proof} First observe that $G_F$ is bipartite. Indeed $G_F = (V_F, E_F)$ has  the bipartition $V_F = V_1 \cup V_2$, where $V_1$ are the source nodes in $G_F$ and  $V_2$ are the destination nodes in $G_F$. 

Turning to the proof,  suppose $G$ has an allowed matching $M$. Define $M'$ to be the collection of edges $(u_{\rm src}, v_{\rm dst})$ in the splitting of $G$ such that $(u,v) \in M$.  Then because $M$ was a matching in $G$, it immediately follows that $M'$ is a matching in $G_F$. Furthermore, because $M$ is an allowed matching in $G$ we have that every node in $V_2$ (i.e., every destination node) is matched in $M'$. Consequently, the size of $M'$ is $|V_1| = |F|$. 

Conversely, given a matching $M'$ in $G_F$ of cardinality $|F|$, observe that since $G_F$ is bipartite and $|V_2| = |F|$, we have that every node in $V_2$ is matched. We then define $M$ to be the collection of edges $(u,v)$ in $G$ such that $(u_{\rm src}, v_{\rm dst})$ or $(v_{\rm dst}, u_{\rm src})$ is in $G_F$ (recall that $G_F$ is undirected). Then because $M'$ is a matching in $G_F$ we obtain that $M$ is a matching in $G$. Moreover, since every node in $V_2$ is matched in $M'$, we get that $M$ is an allowed matching. 

Finally, the last paragraph also tells us how to recover an allowed matching from a matching in $G_F$ of size $|F|$: for each node in $F$, we simply look
at the vertex that it is matched to in $G_F$ and add that as the destination of a matching in $G$. This process takes $O(n)$ operations. 
\end{proof}

An immediate consequence of this proposition is that we may find an allowed matching (or declare that none exists) in $O(m \sqrt{n})$ operations: first $O(m)$ operations to construct the list of edges in $G_F$, then $O(m \sqrt{n})$ operations to run
Hopcroft-Karp with this list of edges as an input to find a maximum matching in $G_F$ and then $O(n)$ additional operations to write down an allowed matching in $G$ (or declare that no allowed matching exists if the size of the maximum matching in $G_F$ is less than $|F|$).

\subsection{The augmentation procedure for minimum cost allowed matchings\label{sec:augment}}

Let us recap our progress thus far. We began in Section \ref{sec:combinatorial} by pointing out that we can spend $O(m+n)$ operations to ensure our graph $G(A)$ has no isolated nodes, which allows us to assume throughout the remainder of the paper that $n \leq 2m$. In Proposition \ref{reformulation}, we showed that a solution to the minimum structural controllability problem 
can be recovered from a minimum cost allowed matching in a certain graph, and the construction of this graph and the recovery of this solution will take $O(m)$ operations. Subsequently, in Proposition \ref{prop:allexists} we observed that an allowed matching in the same graph can be found using the Hopcroft-Karp algorithm in $O(m\sqrt{n})$ additional operations.

 In this section, we further describe an ``augmentation procedure'' which, starting from an allowed matching, finds a minimum cost allowed matching in $O(m \sqrt{n})$ operations. Putting all these results together implies our main finding, which is that the complexity of the minimum structural controllability problem is $O(n + m \sqrt{n})$. 

Before jumping into the algorithm and the proof, we briefly summarize the intution behind our approach. 
It is natural to reformulate the problem of finding a minimum cost allowed matching in terms of flows and explore a Hopcroft-Karp approach of repeatedly augmenting flows along maximal collections of shortest paths (we refer the reader to the original paper \cite{hk} for a detailed explanation of this in the context of finding a maximum cardinality bipartite matching).  

This approach immediately runs into two difficulties. First, one can reduce the cost of a matching without changing how many nodes are unmatched, but rather by shifting the set of unmatched nodes; this means that considerable care is needed the correspondence between flows and matchings, i.e., the usual approach of having all unmatched nodes simply become sink nodes for flow will not work. Secondly, for such a flow-augmentation algorithm to succeed we must ensure  that the flow augmentation procedure we develop does not produce flows corresponding to matchings which are not allowed. Our main finding in this section is that a very careful modification of the usual Hopcroft-Karp approach can bypass these difficulties. 

Stating our algorithm first requires a new slew of definitions, which we now proceed to give. 
\bigskip

\medskip

We will refer to the two matchings in $G$ and the splitting of $G$ which can be obtained from one another by replacing
$(u,v)$ by $(u_{\rm src}, v_{\rm dst})$ and vice versa as {\em twin matchings.} See Figure \ref{fig-twin} and Figure \ref{fig-twin2} for two examples.
\begin{figure}
\begin{tikzpicture}[->, thick]
\SetVertexNormal[LineColor=black]
\SetVertexMath

\node (d) at (-12.5,0) [circle, draw] {$d$};
\node (c) at (-10,0) [circle, draw] {$c$};
\node (b) at (-7.5,0) [circle, draw] {$b$};
\node (a) at (-5,0) [circle, draw] {$a$};

\node (dd) at (3,-2) [circle, draw] {$d_{\rm dst}$};
\node (cd) at (3,0) [circle, draw] {$c_{\rm dst}$};
\node (bd) at (3,2) [circle, draw] {$b_{\rm dst}$};
\node (ad) at (3,4) [circle, draw] {$a_{\rm dst}$};

\node (ds) at (-2,-2) [circle, draw] {$d_{\rm src}$};
\node (cs) at (-2,0) [circle, draw] {$c_{\rm src}$};
\node (bs) at (-2,2) [circle, draw] {$b_{\rm src}$};
\node (as) at (-2,4) [circle, draw] {$a_{\rm src}$};

\path[->, color=red]

(a) edge [bend left] (b)
(b) edge (c)
(b) edge [bend left] (a)
(as) edge (bd)
(bs) edge (ad)
(bs) edge (cd);

\path[->, color=blue]
(a) edge [out=115, in = 55, looseness=4] (a)
(b) edge [out=115, in = 55, looseness=4] (b)
(c) edge (d)
(as) edge (ad)
(bs) edge (bd)
(cs) edge (dd);

\end{tikzpicture}
\caption{On the left, the blue edges constitute a matching. The twin matching in the splitting is shown on the right.}\label{fig-twin}
\end{figure}
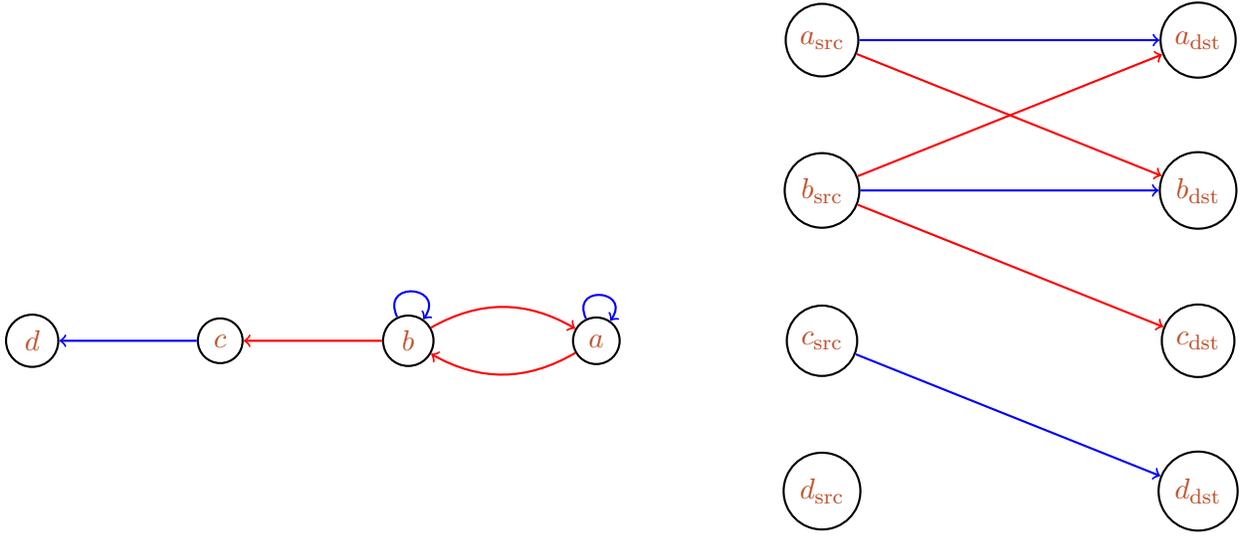 Given a directed graph $G$, we adopt the convention of using $l$ to denote the number of {\em source} strongly  connected components of $G$. The source connected components themselves will be denoted by $S_1, \ldots, S_l$. 

For example, in the graph on the left-hand side of Figure \ref{fig-split},
the only source connected component is $\{a,b\}$, hence $l=1$. The other two connected components, $\{c\}$ and $\{d\}$ are not source. 

 Given a directed graph $G$ and a matching $M$, we will use $X_1, \ldots, X_r$ to denote those source strongly connected components among $S_1, \ldots, S_l$ which do not have an unmatched node and  $Y_1, \ldots, Y_k$ to denote those which have a single unmatched node. Naturally, $r+k \leq l$. We stress that whether a source connected component gets classified as $X_i$ or as $Y_j$ (or neither) depends on the matching $M$. Moreover, we will use 
$Z_1, \ldots, Z_w$ to denote a listing of those among $S_{1}, \ldots, S_l$  which have two or more unmatched nodes {\bf and} all strongly connected components
which are not source. Thus $r+k+w$ is the total number of strongly connected components of $G$.

For example consider the matching shown on the left of Figure \ref{fig-twin}. There is only one source strongly connected component, namely $\{a,b\}$. It does not have any unmatched nodes, so $X_1 = \{a,b\}$. There are no other soure strongly connected component, so we do not have any $Y_i$. Finally, $Z_1 = \{c\}, Z_2 = \{d\}$. 

On the other hand, consider the matching on left of Figure \ref{fig-twin2}. The only source strongly connected component, namely $\{a,b\}$, has a single unmatched node. Thus $Y_1 = \{a,b\}$. There are no variables $X_i$. As in the previous paragraph,
$Z_1 = \{c\}$ and $Z_2 = \{d\}$.

\begin{figure}
\begin{tikzpicture}[->, thick]
\SetVertexNormal[LineColor=black]
\SetVertexMath

\node (d) at (-12.5,0) [circle, draw] {$d$};
\node (c) at (-10,0) [circle, draw] {$c$};
\node (b) at (-7.5,0) [circle, draw] {$b$};
\node (a) at (-5,0) [circle, draw] {$a$};

\node (dd) at (3,-2) [circle, draw] {$d_{\rm dst}$};
\node (cd) at (3,0) [circle, draw] {$c_{\rm dst}$};
\node (bd) at (3,2) [circle, draw] {$b_{\rm dst}$};
\node (ad) at (3,4) [circle, draw] {$a_{\rm dst}$};

\node (ds) at (-2,-2) [circle, draw] {$d_{\rm src}$};
\node (cs) at (-2,0) [circle, draw] {$c_{\rm src}$};
\node (bs) at (-2,2) [circle, draw] {$b_{\rm src}$};
\node (as) at (-2,4) [circle, draw] {$a_{\rm src}$};

\path[->, color=red]
(a) edge [out=115, in = 55, looseness=4] (a)
(b) edge [out=115, in = 55, looseness=4] (b)
(b) edge [bend left] (a)
(bs) edge (ad)
(as) edge (ad)
(bs) edge (bd);

\path[->, color=blue]
(a) edge [bend left] (b)
(b) edge (c)
(c) edge (d)
(as) edge (bd)
(bs) edge (cd)
(cs) edge (dd);

\end{tikzpicture} \caption{On the left, the blue edges constitute a matching. The twin matching in the splitting is shown on the right.}\label{fig-twin2}
\end{figure}
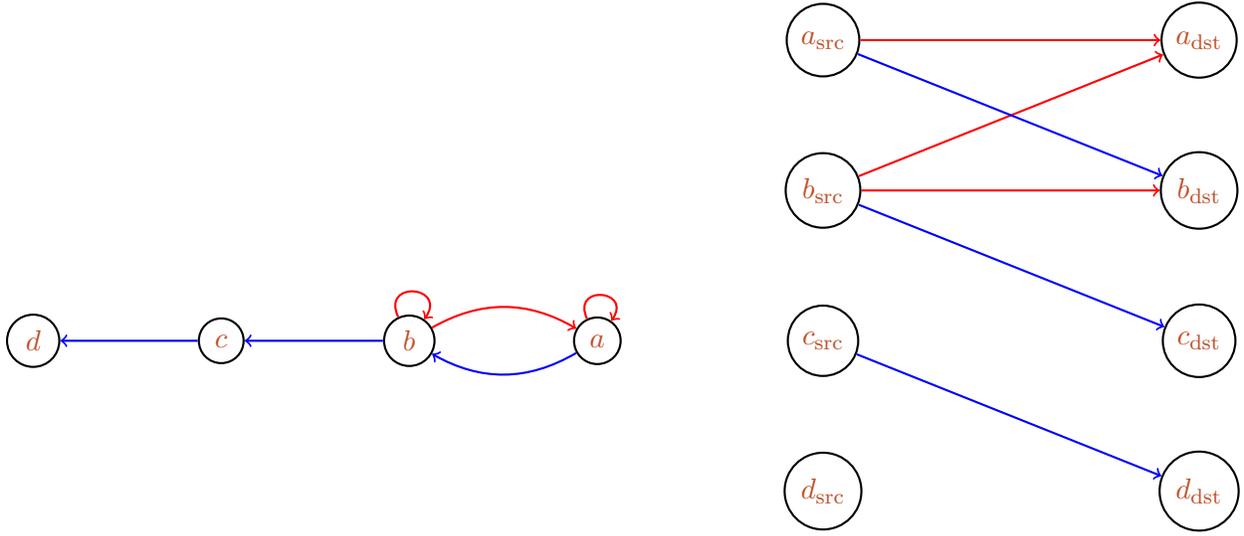 Recall that we have previously defined $U(M)$ to denote the set of unmatched nodes of the matching $M$. We now define $U'(M)$ to denote 
those unmatched nodes which do not lie in some $Y_i$. In other words $U'(M)$ is the set of unmatched nodes which  do not lie in a strongly connected component 
with only one unmatched node.

For the graph on the left-hand side of Figure \ref{fig-twin}, again letting $M$ be the matching consisting of the blue edges, we have
that $U(M) = U'(M) = \{c\}$. On the other hand, for the blue matching on the left-hand side of Figure \ref{fig-twin2}, we have that $U(M)=\{a\}$ but since $a$ lies in $Y_1$, we have that $U'(M)$ is empty.

Given a directed graph $G$ and a subset $A$ of vertices of $G$, we will use $[A]_{\rm dst}$ to denote the set of nodes in 
the splitting of $G$ which are the destination nodes of vertices in $A$. For example, in Figure \ref{fig-split}, $[\{a,b\}]_{\rm dst} = \{a_{\rm dst}, b_{\rm dst}\}$.

The cost of a matching $M$ in the splitting of $G$ is defined to be the number of unmatched destination nodes plus the number of 
$[S_1]_d, [S_2]_d, \ldots, [S_l]_d$ without an unmatched node.  Note that this is exactly equal to the (previously defined) cost of the twin matching of $M$ in $G$. For any
matching $M$ (either in the splitting or in the original graph) we will use $c(M)$ to denote the cost of $M$.

For example, in Figure \ref{fig-twin}, the matching on the left has cost $2$: one for the unmatched node $\{c\}$ and one for the source connected component $\{a,b\}$ without an unmatched node. The matching on the right has cost $2$ as well: one for unmatched $c_{\rm dst}$ and one for $\{a_{\rm dst}, b_{\rm dst} \}$, which is the set of destination nodes of a source connected component without an unmatched node. 

On the other hand, in Figure \ref{fig-twin2}, both matchings have cost $1$ for the unmatched node $\{a\}$ on the left and $a_{\rm dst}$ on the right.

We note that in any matching $M$ in the splitting of $G$, all the source nodes are always unmatched. The cost, however, depends on the number of unmatched destination nodes and may depend on exactly which destination nodes are matched.  Since a matching and its twin matching uniquely determine each other, we will sometimes  find it convenient use them interchangeably. For example given a matching $M$ in the splitting of $G$ and its twin matching $M'$ in 
$G$, we will sometimes say that a node $v$ in $G$ is unmatched under $M$; what is meant is that the node is unmatched under the twin of $M$, i.e., 
under $M'$. Conversely, we might say that the node $v_{\rm dst}$ is unmatched under $M'$; what is meant is that $v_{\rm dst}$ is unmatched under
the twin of $M'$, i.e., under $M$.

We will refer to a matching in the splitting of $G$ as {\em allowed} if every node in $[F]_d$ is matched. Naturally,  a matching in $G$ is an allowed matching if and only if its twin matching in the splitting of $G$ is. 

\bigskip

\medskip

With the above definitions in place, we now proceed to define the notion of a ``flow graph'' which corresponds to a graph $G$ and a matching in it. 

First, given a directed graph $G$ and a matching $M$ in $G$, we define its matched splitting, denoted by $G_{\rm s}(M)$, as follows: first we take the splitting of $G$ and reverse the orientation of every edge in $M$'s twin matching; then for each $i=1, \ldots, k$,  we add edges going from the single unmatched  node in $[Y_i]_d$ to all nodes in $[Y_i]_d \cap F^c$. See Figure \ref{fig-match-split} for an illustration.

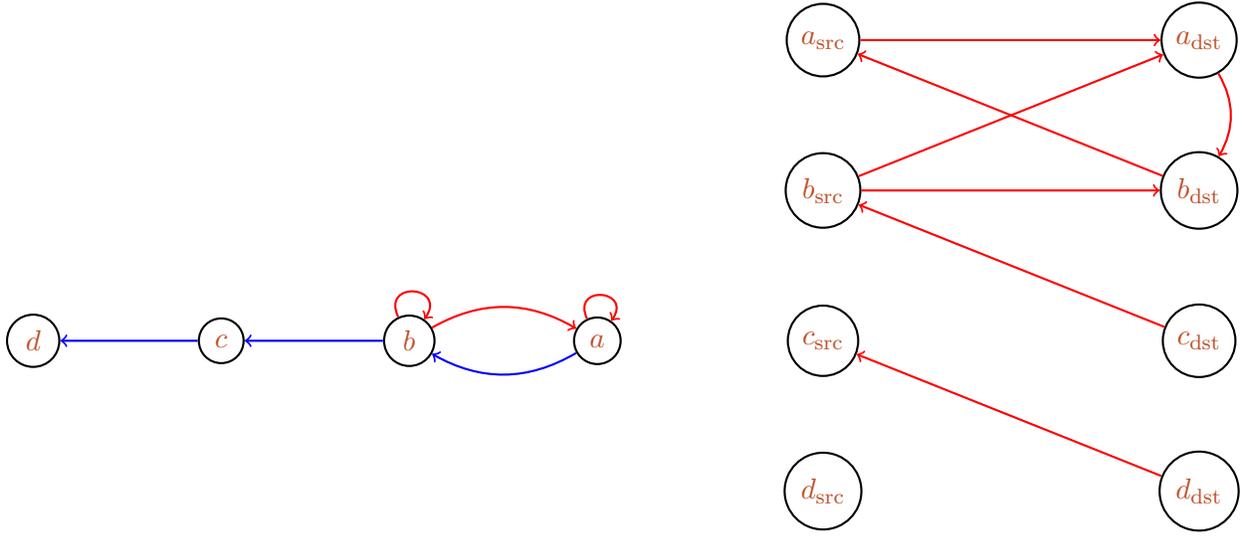
\begin{figure}
\begin{tikzpicture}[->, thick]
\SetVertexNormal[LineColor=black]
\SetVertexMath

\node (d) at (-12.5,0) [circle, draw] {$d$};
\node (c) at (-10,0) [circle, draw] {$c$};
\node (b) at (-7.5,0) [circle, draw] {$b$};
\node (a) at (-5,0) [circle, draw] {$a$};

\node (dd) at (3,-2) [circle, draw] {$d_{\rm dst}$};
\node (cd) at (3,0) [circle, draw] {$c_{\rm dst}$};
\node (bd) at (3,2) [circle, draw] {$b_{\rm dst}$};
\node (ad) at (3,4) [circle, draw] {$a_{\rm dst}$};

\node (ds) at (-2,-2) [circle, draw] {$d_{\rm src}$};
\node (cs) at (-2,0) [circle, draw] {$c_{\rm src}$};
\node (bs) at (-2,2) [circle, draw] {$b_{\rm src}$};
\node (as) at (-2,4) [circle, draw] {$a_{\rm src}$};

\path[->, color=red]
(a) edge [out=115, in = 55, looseness=4] (a)
(b) edge [out=115, in = 55, looseness=4] (b)
(b) edge [bend left] (a)
(bs) edge (ad)
(as) edge (ad)
(bs) edge (bd);

\path[->, color=blue]
(a) edge [bend left] (b)
(b) edge (c)
(c) edge (d);
\path[->, color=red]
(bd) edge (as)
(cd) edge (bs)
(dd) edge (cs)
(ad) edge [bend left] (bd);

\end{tikzpicture} \caption{On the left, the blue edges constitute a matching. Taking $F=\emptyset$, on the right the matched splitting is shown. }\label{fig-match-split}
\end{figure}

Now, given a directed graph $G$ and a matching $M$ in $G$, we define the flow graph $G_{\rm flow}(M)$ as follows. We take 
$G_{\rm s}(M)$ and add vertices $s,t, s_1, \ldots, s_r$ (recall $r$ is the number of connected components of $G$ without an unmatched node in $M$). We add an edge going from $s$ to each source node without an outgoing edge in the twin matching of $M$; from $s$ to each $s_i$; from each $s_i$ to every vertex in $[X_i]_d \cap F^c$. We add links going from every $U'(M)$ which does not lie in a source connected component to $t$. Finally, we go through all the source connected components $S_1, \ldots, S_l$, and letting $k_i$ be the number of unmatched nodes in $S_i$, we create nodes $n^{(i)}_1, \ldots, n^{(i)}_{k_i - 1}$. In other words, we create extra nodes $n^{(i)}_j$ whose number is equal to one less than the number of unmatched nodes in  $S_i$. Finally, for each $i=1, \ldots, l$, we put an edge from each unmatched node in $[S_i]_d$ to all the nodes $n^{(i)}_j$ and from all the nodes $n^{(i)}_j$ to $t$. See Figure \ref{fig-flow-1}, Figure \ref{fig-flow-2}, and Figures \ref{fig-flow-3} \& \ref{fig-flow-4} for illustrations.

\begin{figure}
\begin{tikzpicture}[->, thick]
\SetVertexNormal[LineColor=black]
\SetVertexMath

\node (d) at (-12.5,0) [circle, draw] {$d$};
\node (c) at (-10,0) [circle, draw] {$c$};
\node (b) at (-7.5,0) [circle, draw] {$b$};
\node (a) at (-5,0) [circle, draw] {$a$};

\node (s) at (5,3) [circle, draw] {$s$};
\node (t) at (5,-1) [circle, draw] {$t$};

\node (dd) at (1,-2) [circle, draw] {$d_{\rm dst}$};
\node (cd) at (1,0) [circle, draw] {$c_{\rm dst}$};
\node (bd) at (1,2) [circle, draw] {$b_{\rm dst}$};
\node (ad) at (1,4) [circle, draw] {$a_{\rm dst}$};

\node (ds) at (-3,-2) [circle, draw] {$d_{\rm src}$};
\node (cs) at (-3,0) [circle, draw] {$c_{\rm src}$};
\node (bs) at (-3,2) [circle, draw] {$b_{\rm src}$};
\node (as) at (-3,4) [circle, draw] {$a_{\rm src}$};

\path[->, color=red]
(a) edge [out=115, in = 55, looseness=4] (a)
(b) edge [out=115, in = 55, looseness=4] (b)
(b) edge [bend left] (a)
(bs) edge (ad)
(as) edge (ad)
(bs) edge (bd);

\path[->, color=blue]
(a) edge [bend left] (b)
(b) edge (c)
(c) edge (d);
\path[->, color=red]
(bd) edge (as)
(cd) edge (bs)
(dd) edge (cs)
(ad) edge [bend left] (bd)
(s) edge [bend left] (ds);

\end{tikzpicture} \caption{On the left, the blue edges constitute a matching. On the right, taking $F=\emptyset$, the flow graph is shown. Note that because $r=0$ (namely, there are no strongly connected components in the original graph without unmatched nodes) the nodes $s_i$ are not present. Moreover, because the only source strongly connected component has a single unmatched node, no nodes $n^{(i)}_j$ are present.}\label{fig-flow-1}
\end{figure}
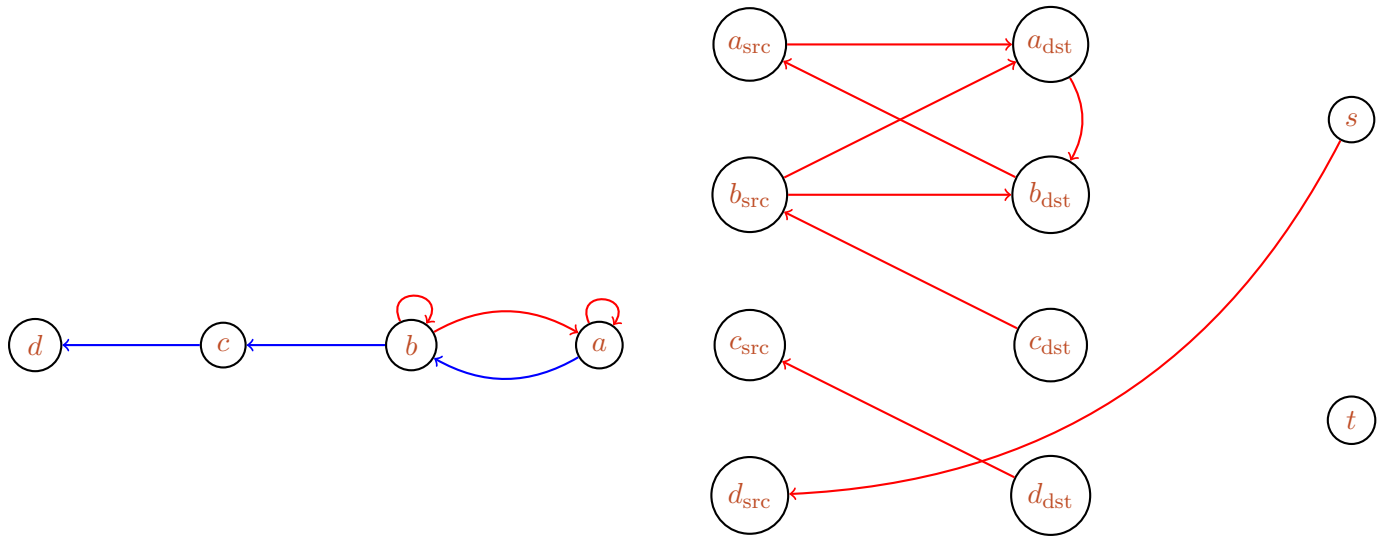

\begin{figure}
\begin{tikzpicture}[->, thick]
\SetVertexNormal[LineColor=black]
\SetVertexMath

\node (d) at (-12.5,0) [circle, draw] {$d$};
\node (c) at (-10,0) [circle, draw] {$c$};
\node (b) at (-7.5,0) [circle, draw] {$b$};
\node (a) at (-5,0) [circle, draw] {$a$};

\node (s) at (5,3) [circle, draw] {$s$};
\node (s1) at (3,3) [circle, draw] {$s_1$};
\node (t) at (5,-1) [circle, draw] {$t$};

\node (dd) at (1,-2) [circle, draw] {$d_{\rm dst}$};
\node (cd) at (1,0) [circle, draw] {$c_{\rm dst}$};
\node (bd) at (1,2) [circle, draw] {$b_{\rm dst}$};
\node (ad) at (1,4) [circle, draw] {$a_{\rm dst}$};

\node (ds) at (-3,-2) [circle, draw] {$d_{\rm src}$};
\node (cs) at (-3,0) [circle, draw] {$c_{\rm src}$};
\node (bs) at (-3,2) [circle, draw] {$b_{\rm src}$};
\node (as) at (-3,4) [circle, draw] {$a_{\rm src}$};

\path[->, color=red]

(a) edge [bend left] (b)
(b) edge (c)
(b) edge [bend left] (a)
(as) edge (bd)
(bs) edge (ad)
(bs) edge (cd)
(s) edge (s1)
(s1) edge (ad)
(s1) edge (bd)
(cd) edge (t);

\path[->, color=blue]
(a) edge [out=115, in = 55, looseness=4] (a)
(b) edge [out=115, in = 55, looseness=4] (b)
(c) edge (d);
\path[->, color=red]
(ad) edge (as)
(bd) edge (bs)
(dd) edge (cs)
(s) edge [bend left] (ds);

\end{tikzpicture}
\caption{On the left, the blue edges constitute a matching. Taking $F=\emptyset$, the flow graph is shown on the right.  Note that because the only source strongly connected component has no unmatched nodes, no nodes $n^{(i)}_j$ were added.}\label{fig-flow-2}
\end{figure}
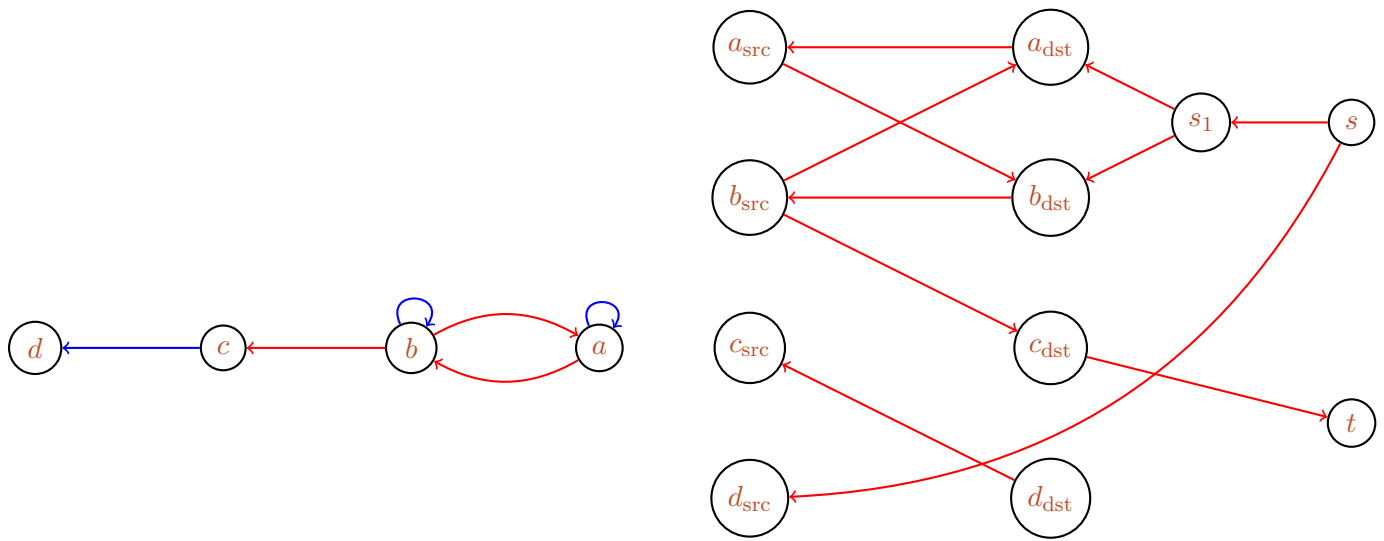

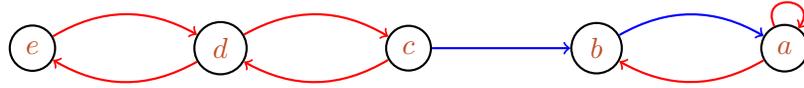
\begin{figure} \begin{center}
\begin{tikzpicture}[->, thick]
\SetVertexNormal[LineColor=black]
\SetVertexMath

\node (e) at (-15,0) [circle, draw] {$e$};
\node (d) at (-12.5,0) [circle, draw] {$d$};
\node (c) at (-10,0) [circle, draw] {$c$};
\node (b) at (-7.5,0) [circle, draw] {$b$};
\node (a) at (-5,0) [circle, draw] {$a$};

\path[->, color=red]

(a) edge [bend left] (b);

\path[->, color=blue]
(c) edge (b)
(b) edge [bend left] (a);

\path[->, color=red]
(a) edge [out=115, in = 55, looseness=4] (a)
(d) edge [bend left] (c)
(c) edge [bend left] (d)
(d) edge [bend left] (e)
(e) edge [bend left] (d);

\end{tikzpicture}
\caption{The blue edges constitute a matching. Assuming $F = \emptyset$, he flow graph of this is shown in Figure \ref{fig-flow-4}.}\label{fig-flow-3} \end{center}
\end{figure}

\begin{figure} \begin{center}
\begin{tikzpicture}[->, thick]
\SetVertexNormal[LineColor=black]
\SetVertexMath

\node (s) at (5.5,3) [circle, draw] {$s$};
\node (t) at (9,-1) [circle, draw] {$t$};
\node (n1) at (7,-3) [circle, draw] {$n^{(1)}_2$};
\node (n2) at (7,-1) [circle, draw] {$n^{(1)}_1$};

\node (ed) at (1,-4) [circle, draw] {$e_{\rm dst}$};
\node (dd) at (1,-2) [circle, draw] {$d_{\rm dst}$};
\node (cd) at (1,0) [circle, draw] {$c_{\rm dst}$};
\node (bd) at (1,2) [circle, draw] {$b_{\rm dst}$};
\node (ad) at (1,4) [circle, draw] {$a_{\rm dst}$};

\node (es) at (-3,-4) [circle, draw] {$e_{\rm src}$};
\node (ds) at (-3,-2) [circle, draw] {$d_{\rm src}$};
\node (cs) at (-3,0) [circle, draw] {$c_{\rm src}$};
\node (bs) at (-3,2) [circle, draw] {$b_{\rm src}$};
\node (as) at (-3,4) [circle, draw] {$a_{\rm src}$};

\path[->, color=red]

(as) edge (bd)
(ad) edge (bs)
(cd) edge (bs)
(cs) edge (dd)
(ds) edge (ed)
(ds) edge (cd)
(es) edge (dd)
(n1) edge (t)
(n2) edge (t)
(cd) edge (n1)
(cd) edge (n2)
(dd) edge (n1)
(dd) edge (n2)
(ed) edge (n1)
(ed) edge (n2)
(s) edge [bend right] (as);

\path[->, color=red]
(as) edge (ad)
(bd) edge (cs)
(s) edge [bend left] (ds)
(s) edge [bend left] (es);

\end{tikzpicture}
\caption{The flow graph of the graph and matching from Figure \ref{fig-flow-3}. Note that because there are no strongly connected components with no unmatched nodes, the nodes $s_i$ are not present. }\label{fig-flow-4} \end{center}
\end{figure}
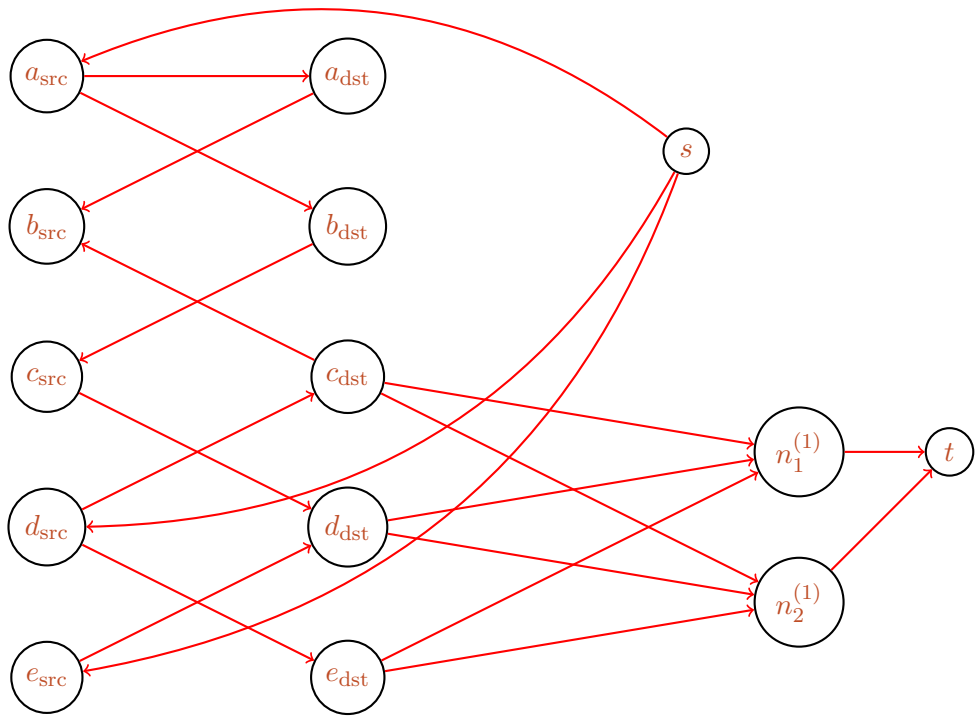

Since a matching uniquely determines its twin matching, we will  also use the notation $G_{\rm flow}(M)$ also for matchings $M$ in the splitting of $G$.  Given a directed graph $G$ and a matching $M$, we will refer to all edges in $G_{\rm s}(M)$ that go from a source to a 
destination node or vice versa as {\em core edges}. By construction, the only edges in $G_{\rm s}(M)$ which are not core edges are those that
go from the single unmatched vertex in each $[Y_i]_d$ to the nodes in $[Y_i]_d \cap F^c$. We will refer to core edges in $G_{\rm s}(M)$ also as core edges in $G_{\rm flow}(M)$. Note that the set of non-core edges in $G_{\rm flow}(M)$ 
additionally includes all edges incident upon $s,t$, and one of the $s_i$ or $n^{(i)}_j$. Finally, we will refer to any source or destination node in $G_{\rm flow}(M)$ as {\em core nodes}. Note that $s,t$ and the nodes $s_i, n^{(i)}_j$ are the only nodes in $G_{\rm flow}$ which are not core nodes.

\subsubsection{The augmentation procedure} 

Suppose $M$ is a matching in the splitting of $G$ and consider the corresponding graph $G_{\rm flow}$. Let $p$ be either:
\begin{enumerate} \item A cycle in $G_{\rm flow}$.
\item A path in $G_{\rm flow}$ which does not begin 
at a source node with an outgoing link in $M$ and does not end at a destination node which is matched in $M$.
\end{enumerate}

The augmentation of $M$ on $p$ is the collection of edges $M'$ defined as follows: we remove from $M$ the reverse of every edge in $p$ that goes from a destination node to a source node\footnote{Recall that edges in the matching $M$ had their direction reversed in the matched splitting $G_{\rm s}(M)$, and consequently appear reversed in $G_{\rm flow}(M)$.} and adding every edge in $p$ that goes from a source node to a destination node. 

If ${\cal P}$ is a collection of vertex-disjoint paths (satisfying item (2) above) and cycles in $G_{\rm flow}$, then the augmentation of $M$ on ${\cal P}$ is obtained by sequentially augmenting on each path and cycle in $p$,
in arbitrary order.

For example, consider the flow graph of Figure \ref{fig-flow-1}. Consider the twin matching of the matching in blue; this is the matching shown on the right in Figure \ref{fig-twin2}. Consider performing the augmentation with the cycle $(a_{\rm dst}, b_{\rm dst}), (b_{\rm dst}, a_{\rm src}), (a_{\rm src}, b_{\rm dst})$ in the graph $G_{\rm flow}$. The only core edges in this length-three cycle are 
$(b_{\rm dst}, a_{\rm src})$ and $(a_{\rm src}, b_{\rm src})$. Performing the augmentation 
procedure therefore removes the edge $(a_{\rm src}, b_{\rm dst})$ from the matching and 
adds instead the edge $(a_{\rm src}, a_{\rm dst})$. Note that the result is a matching in which $a_{\rm dst}$ has become matched 
and $b_{\rm dst}$ has become unmatched. 

It turns out that if $M$ is a matching in the splitting of $G$, then so is $M'$. However, we will not prove this now (this is Lemma \ref{lemma:allowed} below).

\subsubsection{The algorithm for finding a minimum cost allowed matching\label{sec:alg}}

Having introduced the definition of the graph $G_{\rm flow}(M)$ and described the augmentation procedure, we  now turn to the 
description of our algorithm for finding minimum cost allowed matchings. We assume that the algorithm is initialized at some allowed matching $M$ in $G$.

We first need to define the notion of a {\em maximal vertex-disjoint collection of shortest paths in a directed graph} (which is actually more-or-less self explanatory). A collection ${\cal S}$ of paths 
from $a$ to $b$ in some directed graph is a maximal vertex-disjoint collection of shortest paths from $a$ to $b$ if (i) each path in ${\cal S}$ is a shortest path from $a$ to $b$ (ii) any two paths in ${\cal S}$ have only the nodes $a$ and $b$ in common (iii) it is not possible to add another path to ${\cal S}$ and still satisfy (i) and (ii). With this in mind, our algorithm is as follows; we will refer to it as {\em the augmentation algorithm}.

\begin{algorithm} \begin{enumerate} \item Construct $G_{\rm flow}(M)$. 
\item Find a maximal collection of vertex-disjoint shortest paths from $s$ to $t$ in $G_{\rm flow}(M)$. Call this collection ${\cal P}$. 
\item  If ${\cal P}$ is empty (i.e., if there is no path from $s$ to $t$ in $G_{\rm flow}$) then terminate. \\
Else replace $M$ with the augmentation of $M$ on ${\cal P}$.
\end{enumerate}
\end{algorithm}

Note the similarity of this algorithm to the standard Hopcroft-Karp method. It remains to analyze the performance of the algorithm. We will prove the following two facts. 

\begin{proposition} Each iteration of this algorithm [i.e., each execution of the steps (1), (2), (3)] can be implemented in $O(m)$ operations. 
\label{prop:opcount}
\end{proposition} 

\begin{theorem} The algorithm terminates after $O(\sqrt{n})$ iterations with $M$ an allowed matching of minimum cost. \label{thm}
\end{theorem}

\bigskip

Of course, putting these two facts together immediately implies that the minimum cost allowed matching is found by the algorithm after
$O(m \sqrt{n})$ steps.

Proposition \ref{prop:opcount} is easy to show and we will launch into its proof shortly. By contrast, Theorem \ref{thm} will  require
considerable exertions on our part and will only be proved at the end  of Section \ref{sec:analysis}.  We remark that Theorem \ref{thm} is the main technical result of this paper and seems to be considerably more difficult to show than the analogous statement shown by Hopcroft and Karp \cite{hk} in the context of bipartite matchings. 

We conclude Section \ref{sec:alg} by proving Proposition \ref{prop:opcount}.

\begin{proof}[Proof of Proposition \ref{prop:opcount}] First, we argue that the graph $G_{\rm flow}(M)$  takes $O(m)$ operations to construct. Indeed, adding edges between source and destination nodes clearly takes $O(m)$ time. As we have previously described, we can use Kosaraju's algorithm to compute the strongly connected components of $G$ in  $O(m)$ time, and the remainder of the graph can be constructed in $O(m)$ time by inspection. The only issue is that we can have as many as $\Omega(n^2)$ edges going between destination nodes and various nodes $n^{(i)}_j$.  We bypass this by never writing down these edges directly in memory but rather just maintaining a table of unmatched destination nodes and their corresponding nodes $n^{(i)}_j$, if any. 

The fact that a maximal collection of vertex-disjoint shortest paths in a directed graph can be found 
in $O(\mbox{ number of edges} )$ time is standard, see for example the exposition in Section VIII.4 of \cite{comptop}.  We briefly summarize. We  first find all the edges which are on a shortest path from $s$ to $t$. These form a directed acyclic graph, and we keep track of the in-degree of each node in this graph. We then trace back a path from $t$ to $s$, delete the vertices on this path, adjust the in-degrees, and recursively delete all the nodes whose in-degree becomes zero as a result. We then trace back another path from $t$ to $s$ and so forth.

 We want to further argue this can be implemented in $O(m)$ time on $G_{\rm flow}(M)$, even though $G_{\rm flow}(M)$ can have $\Omega(n^2)$ edges going from destination nodes to nodes $n^{(i)}_j$. But indeed, we never need to write down these edges explicitly. At each stage of the above algorithm, we can simply
maintain a list of the destination nodes in each $[S_l]_d$ which are on a shortest path and have not been deleted and likewise of all the $n^{(i)}_j$ corresponding to $S_l$ which are on
the shortest path and have not been deleted, with the understanding that there is a outgoing link from each such node in $[S_l]_d$ to each such $n^{(i)}_j$.

Finally, augmenting on ${\cal P}$ (which has at most $O(m)$ edges in it) requires going through each edge in ${\cal P}$ and either deleting or 
adding it to the matching, and takes $O(m)$ operations as well. 
\end{proof}

\subsubsection{Analysis\label{sec:analysis}}

We now turn to the analysis of the augmentation procedure, culminating in the proof of Theorem \ref{thm}. We will need to carefully analyze what happens to each edge as we augment.  


\begin{lemma} Suppose $M$ is a matching in the splitting of $G$ and $M'$ is the collection of edges obtained after performing the augmentation procedure. Then:
\begin{itemize} \item $M'$ is a matching in the splitting of $G$.  
\item If $p$ is a path from $s$ to $t$, then $M'$ has strictly lower cost than $M$.
\item If $p$ is a cycle, then $M'$ has the same cost as $M$.
\item If $M$ is an allowed matching and $p$ is a cycle or a path which does not begin from a node in $[F]_d$, then $M'$ is an 
allowed matching as well.
\end{itemize} \label{lemma:allowed}
\end{lemma} 

\begin{proof} To show that $M'$ is a valid matching we must show two things: that no two of its edges leave the same node and that no two of its edges arrive at the same node. 

 We first argue that no two edges in $M'$ leave the same node. Indeed, observe that if the augmentation procedure adds  $(c,d)$ to $M'$ then $c$ must be a source node while $d$ must be a destination node; and moreover, one of the following conditions holds:
\begin{enumerate} \item $p$ is a path and $c$ is its first core node. Then $c$ has no outgoing link in $M$.
\item  Either $p$ is not a path, or $p$ is a path but $c$ is not its first core node. Then $c$ is matched in $M$ and
the augmentation on $p$ has  previously removed the edge $(c,e)$ from the matching for some $e$.
\end{enumerate} Consequently, no two edges leave the same node in $M'$. 

To see that no two edges are incoming on the same node in $M'$, suppose the augmentation adds $(c,d)$ to the matching. If $d$ is unmatched in $M$, there is no problem; and if $d$ is matched but $p$ is a cycle, this means we have previously removed some edge $(f,d)$ from $M$ so that, once again, there is no problem. Finally if $d$ is matched in $M$ and $p$ is a path, by assumption we have that $d$ is not the last node of the path $p$ (since the augmentation procedure is only performed on paths which do not end at matched destination nodes). Furthermore, by construction of $G_{\rm flow}$, we can further conclude that the next edge in $p$ is a core edge; indeed, non-core edges go out only from unmatched destination nodes. Thus once again we know that we have previously removed the  edge $(f,d)$ from the matching for some $f$. We conclude that no two edges in $M'$ are incoming on the same node. Thus  $M'$ is a matching. 

Now suppose $p$ is a path from $s$ to $t$ and let us consider how the set of matched destination nodes changes after performing the augmentation procedure. There are two possibilities. 

First, let us consider the case when the first edge in $p$ leads to an unmatched source node. Now each time the path takes the edge $(a_{\rm dst}, b_{\rm dst})$, we have that $a_{\rm dst}$ goes from unmatched-to-matched after the augmentation while $b_{\rm dst}$ goes from unmatched-to-matched. No other destination nodes besides all such $b_{\rm dst}$ become unmatched as a result of the augmentation
procedure, and the last core (destination) node on the path becomes matched. Since $a_{\rm dst}$ and $b_{\rm dst}$ belong to the same 
$[Y_i]_d$, we have that the cost strictly decreases by one.

Second, suppose the first edge in $p$ leads to some $s_i$, and the second edge in $p$ then leads to a matched destination node. This destination node goes from being matched in $M$ to unmatched in $M'$; and the destination node which is the {\em last} destination node goes from unmatched in $M$ to matched in $M'$. Moreover, each time  the non-core edge $(a_{\rm dst}, b_{\rm dst})$ appears in $p$
we again have that $a_{\rm dst}$ goes from being unmatched in $M$ to matched in $M'$, while $b_{\rm dst}$ goes from being matched in $M$ to being unmatched in $M'$. Again in this case both $a$ and $b$ belong to the same $[Y_i]_d$. Consequently, after performing the augmentation on the path $p$ under consideration, $[Y_i]_d$ still has a single unmatched node. 

Now recall that, by construction of $G_{\rm flow}$, the first destination node on the path must lie in some $[X_i]_d$ while the last destination node on the path
lies in $[U'(M)]_d$. Thus performing the augmentation on a path $p$ results in: (i) the set $[X_i]_d$ in which the first destination node of $p$ lies goes from having no unmatched node to a single unmatched node (ii) the set of unmatched nodes in any $[Y_i]_d$ is unaffected (iii) the set of unmatched node in some $[Z_i]_d$ decreases by one. It immediately follows that doing the augmentation procedure on $p$ decreases cost by one.

The proof that augmenting on a cycle does not reduce cost is essentially identical to the proof we have just given for paths: it is an immediate corollary of the observation that the only change in the set of matched nodes comes as a result of the edges $(a_{\rm dst}, b_{\rm dst})$, which has the consequence of changing which node within some $[Y_j]_d$ is matched with no effect on cost. We omit the details.

Finally, supposing that $M$ was an allowed matching, let us show that the matching $M'$ we get after the augmentation procedure is also an allowed matching, under the assumption that $p$ is a cycle or a path that does not begin at a node in $[F]_d$. Indeed, let us consider all the ways in which the augmentation procedure can lead to a node becoming unmatched. As we have remarked earlier, when the augmentation procedure goes through $(a_{\rm dst}, b_{\rm dst})$ then $b_{\rm dst}$ goes from matched-to-unmatched. Also, if $c_{\rm dst}$ is the first core node on a path, then $c_{\rm dst}$ goes from matched-to-unmatched. In no other case is a destination node made unmatched by the augmentation procedure. Since by construction, nodes $b_{\rm dst}$ such that the 
edge $(a_{\rm dst}, b_{\rm dst})$ appears in $p$ are not in $F$, and by assumption nodes $c_{\rm dst}$ which are the first nodes of any path are also not in $F$,  we have that $M'$ is an allowed matching. 
\end{proof}

Along similar lines, the following lemma lists some things that do not happen during the augmentation procedure. 

\begin{lemma} Suppose $M$ is a matching and let $M'$ be the result of performing an iteration of the augmentation algorithm (i.e., of performing the augmentation procedure on a maximal collection of vertex-disjoint shortest paths in $G_{\rm flow}(M)$). Then:
\begin{enumerate} \item If $a_{\rm src}$ has an outgoing edge in $M$, it has an outgoing edge in $M'$. 
\item If some source connected component has a single unmatched node or no unmatched nodes in $M$ (i.e., if it is some $X_i$ or some $Y_i$) then it has at most one unmatched node in $M'$. 
\item  No source connected component goes from having at least one unmatched node in $M$ to having no unmatched nodes in $M'$. 
\end{enumerate} \label{lemma:donot}
\end{lemma}

\begin{proof} Let ${\cal P}$ be a maximal collection of vertex-disjoint shortest paths from $s$ to $t$ in $G_{\rm flow}(M)$. Item (1) follows because the augmentation on ${\cal P}$ always adds an outgoing edge from each $a_{\rm src}$ whose outgoing edge it removes. Item (2) follows because augmenting on ${\cal P}$ does not change the number of unmatched nodes in any $[Y_i]_d$ and can  increase the number of unmatched nodes in any $[X_i]_d$ by at most $1$. Finally, item (3) follows because, due to the presence of the nodes $n^{(i)}_j$, if the source connected component $S_i$ has $k_i$ unmatched nodes, then at most $k_i-1$ paths in ${\cal P}$ have 
destination nodes in $[S_i]_d$ as their last core nodes. Consequently, $S_i$ retains at least one unmatched node after augmenting on ${\cal P}$. 
\end{proof}

We now need to define some new notions to proceed.  Suppose $M$ and $M'$ are matchings in the splitting of $G$; the graph $\Delta(M,M')$ is a subgraph of the splitting $G_{\rm s}(M)$ obtained the  by taking the symmetric diference of $M$ and $M'$ and then reversing the orientation of every edge of 
$M$. We then have the following lemma.

\begin{lemma} If $M$ and $M'$ are matchings, then each weakly connected component of $\Delta(M,M')$ is either a path or a cycle. Furthermore, no weakly connected component of ${\Delta}(M,M')$ which is a path can begin at a source node which has an outgoing link in $M$ or end at a destination node which is matched in $M$. \label{lemma:deltastructure}
\end{lemma}

\begin{proof}  We first argue that for any node $a$, there cannot be two edges in ${\Delta}(M,M')$ incoming on $a$ and there cannot be two edges in $\Delta(M,M')$ outgoing from $a$. 

Indeed, suppose $a$ is a source node with two (or more) out-going edges; then both outgoing edges would have to lie in $M'$, since all outgoing edges from source nodes in $\Delta(M,M')$ come from $M'$. This would contradict the fact that $M'$ is a matching. If $a$ is a source node with two incoming edges, both of them would have to come from $M$, which contradicts the fact that $M$ is a matching. The case when $a$ is a destination node is similar. This proves that each weakly connected component of $\Delta(M,M')$ is a path or a cycle. 

To prove the second part, suppose that $p$ is a path weakly connected component of $\Delta(M,M')$ which begins at source node $a$. If $a$ has an outgoing link in $M$, then because no incoming edges on $a$ are in $p$, we can conclude that the outgoing edge from $a$ was the same in $M$ and $M'$. But this means that $a$ has no incoming or outgoing edges in $\Delta(M,M')$ and thus cannot be the start of a path. A similar argument shows the last node of a path cannot be a matched destination node.  \end{proof}

Since all the weakly connected components of $\Delta(M,M')$ are by definition vertex-disjoint, the previous lemma implies we can perform the 
augmentation procedure on them. We then immediately have the following corollary. 

\begin{corollary} Suppose $M$ and $M'$ are matchings in the splitting of $G$ and consider performing the augmentation procedure on $\Delta(M,M')$. The final result is $M'$. \label{cor:ma}
\end{corollary}

\begin{proof} Indeed, doing the augmentation procedure on $\Delta(M,M')$ removes all the edges which are in
$M$ but not in $M'$ and adds all the edges in $M'$ not in $M$.
\end{proof}

We will also require the following lemma, which tells us something about the structure of $\Delta(M,M')$ when both $M$ and $M'$ are  allowed matchings. 

\begin{lemma} Let $M$ and $M'$ be allowed matching. Then no weakly connected component of $\Delta(M,M')$ is a path beginning at a node in
$[F]_d$.  \label{lemma:fsymmdiff}
\end{lemma} 

\begin{proof} Suppose otherwise: there is some $f \in F$ such that $\Delta(M,M')$ has a path beginning at $f_{\rm dst}$. This means
$f_{\rm dst}$ has no incoming edge in $\Delta(M,M')$. This implies that either of the the following two possibilities holds: 
(i) $f_{\rm dst}$ has no incoming edge in $M'$ (ii) $f_{\rm dst}$ has an incoming edge in $M'$, and it is the same edge as its incoming edge in $M$. Each of these two possibilities leads to a contradiction: (i) contracts $M'$ being an allowed matching and (ii) implies that $f_{\rm dst}$ has no outgoing edges in $\Delta(M,M')$ and thus cannot be the start of a path.  
\end{proof}

We are now ready to complete a key step in the in the proof of Theorem \ref{thm}. The rather involved construction of $G_{\rm flow}$ as well as all the previous definitions have been written in order that the following lemma is holds. 

\begin{lemma} Let $M$ be an allowed matching in $G$. Recall that $c(M)$ denotes the cost of $M$ and let $c^*$ be the minimum cost of any allowed matching in $G$. We then have that there exist $c(M)-c^*$ vertex-disjoint paths from $s$ to $t$ in $G_{\rm flow}(M)$.  \label{lemma:pathscount}
\end{lemma} 

\begin{proof}  Let $M^*$ be an allowed matching in $G$ of cost $c^*$ and let 
\[  p_1, \ldots, p_{h'} \] be a listing of the paths in $\Delta(M, M^*)$.  Let us greedily ``clump'' these paths as follows: if some path $p_i$ has a last node $a_{\rm dst}$ in some $[Y_i]_d$ and another path $p_j$ has a first node $b_{\rm dst}$ in the same $[Y_i]_d$, we merge them into the same path by inserting the non-core edge $(a_{\rm dst}, b_{\rm dst})$. Note that this edge $(a_{\rm dst},b_{\rm dst})$ exists since, first by Lemma \ref{lemma:deltastructure}, the node $a_{\rm dst}$ is unmatched; and by Lemma \ref{lemma:fsymmdiff} we have that $b_{\rm dst} \notin F$;  and finally, by construction, $G_{\rm flow}(M)$ has edges leading from the single unmatched node in each $[Y_i]_d$ to all nodes in $[Y_i]_d \cap F^c$. 

 We continue greedily clumping these paths in this way until we cannot clump any more. Let us call the result \[ q_1, \ldots, q_{h''} \] Observe that performing the augmentation procedure on
$q_1, \ldots, q_{h''}$ is the same as performing the augmentation procedure on $p_1, \ldots, p_{h'}$ since the augmentation procedure automatically ignores the extra non-core edges we inserted. 

Let us observe several key properties of the paths $q_1, \ldots, q_{h''}$. First, they are naturally vertex disjoint. Secondly, if some $q_i$ has a final node in some $[Y_j]_d$, no other $q_i$ has a 
starting node in the same $[Y_j]_d$ (because else they would have been clumped together). Similarly, if some $q_i$ begins in
some $[Y_j]_d$, then no other $q_i$ has last node in the same $[Y_j]_d$.

Now since augmenting on a cycle does not change the cost of a matching (by Lemma \ref{lemma:allowed}), we have by Corollary \ref{cor:ma} that the augmentation procedure on $q_1, \ldots, q_{h''}$ decreases the cost of the matching by $c(M)-c^*$. Order $q_1, \ldots, q_{h''}$ arbitrarily and let $\widehat{q}_1, \widehat{q}_2, \ldots$ be a listing of all the $q_i$ which result in cost decreases when doing the augmenting procedure on them sequentially. Since augmenting on any path can decrease cost by at most $1$, we have that there are at least $c(M)-c^*$ paths $\widehat{q}_i$ in this list. 

 We next argue that at least half of these paths $\widehat{q}_i$ can be extended to a path from $s$ to $t$ by appending at the beginning of the path either the edge 
$(s, a_{\rm src})$ for some source node $a_{\rm src}$ or the edges $(s,s_i), (s_i, x_{\rm dst})$ for some $x_i$ and some $x_{\rm dst}$ belonging to some $[X_i]_d$; and appending an edge either going directly to $t$ or going to some $n^{(k)}_j$ and then to $t$ at the end of the path. Furthermore, we will argue that appending these edges keeps the 
paths $\widehat{q}_i$ vertex disjoint. Once these claims are shown, the proof of the current lemma will be complete.

Indeed, by Lemma \ref{lemma:deltastructure} the paths $\widehat{q}_i$ must begin at a source node with no outgoing links in $M$ or a matched destination node in $M$. In the first case, we can append a link going from $s$ to the first node of the path. In the second case, we further have that first node of the path belongs to some $[X_j]_d$ - else, augmenting on the path would not result in a
cost decrease. In this case, we can append a link going from $s$ to $s_i$, and $s_i$ to the matched destination node which is the first on the path; this is possible because Lemma \ref{lemma:fsymmdiff} implies that the first node on the path is in $F^c$, and by construction $s_i$ has outgoing links to every node in $[X_j]_d \cap F^c$. 

Furthermore, we argue that this addition of initial edges to all the $\widehat{q}_i$ keeps them vertex disjoint. Indeed, this is immediate in the case when the first node of
$\widehat{q}_i$ is a source node. In the case when the first node of $\widehat{q}_i$ is a matched destination node, we observe that we cannot have two of the paths $\widehat{q}_k$ beginning at matched destination nodes in the same $[X_j]_d$, since augmenting on one of them after having augmented on the other could not decrease cost. Consequently, adding the links $(s,s_i)$ and from $s_i$ to the first node of the path $\widehat{q}_i$ keeps 
the paths vertex disjoint.

As to the end of these paths, recall that by Lemma \ref{lemma:deltastructure}  each path $\widehat{q}_i$ must end at either a source node or an unmatched destination node in $M$. Since augmentation on any path ending at a source node cannot decrease cost, it must be the latter. Consider the last node of the path, which we will call $d_i$. We argue it does not lie in some $[Y_j]_d$. Indeed, after we have augmented on $\widehat{q}_1, \ldots, \widehat{q}_{i-1}$, we trivially have that $d_i$ does not lie in a source connected component with one unmatched destination node because the augmentation on $\widehat{q}_i$ decreases cost. 
We need to argue this holds initially, before the augmentations on $\widehat{q}_1, \ldots, \widehat{q}_{i-1}$ are done. 

Indeed, if $d_i$ did lie in some $[Y_j]_d$, we claim  that the destination nodes comprising $[Y_j]_d$ would have remained with a single unmatched destination node after we have augmented on $\widehat{q}_1, \ldots, \widehat{q}_{i-1}$. This is because the node $d_i$ was initially unmatched  and we know that there have been no paths in $\widehat{q}_1, \ldots, \widehat{q}_{i-1}$ whose starting points were in $[Y_j]_d$: any such path would have been clumped with $\widehat{q}_i$. Thus, under the assumption that $d_i$ lies in some $[Y_j]_d$, we have that no node in $[Y_j]_d$ has become unmatched after augmentation on $\widehat{q}_1, \ldots, \widehat{q}_{i-1}$, while the only unmatched node has remained unmatched. This proves that if $d_i$ lies in some $[Y_j]_d$ initially, then it still lies in a source strongly connected component with a single unmatched node by the time we augment on $\widehat{q}_1, \ldots, \widehat{q}_{i-1}$. Since we know that augmenting on $\widehat{q}_i$ reduces cost, we therefore have that $d_i$ lies in some $[Z_j]_d$. 

Now if $d_i$ does not belong to the set of destination nodes of a source connected component, we append the edge $(d_i,t)$
and we are done. If $d_i$ does lie in a source connected component, then we append a link going from $d_i$ to any $n^{(k)}_j$ out-neighbor which does not already have an incoming link from doing the appending to  $\widehat{q}_1, \ldots, \widehat{q}_{i-1}$. 

We finally need argue such a neighbor $n^{(k)}_j$ will always be available. Indeed, if a source connected component $S_j$ has $k_j$ unmatched nodes, then there will be at most $k_j-1$ paths $\widehat{q}_i$ with a destination node in $[S_j]_d$: otherwise augmenting on the $k_j$'th such $\widehat{q}_i$ would not result in a cost decrease. Since there are $k_j-1$ nodes $n^{(k)}_j$ with incoming links from this
connected component, one will always be available. \end{proof}

\bigskip

\medskip

As mentioned, the previous lemma is a key step in the proof of Theorem \ref{thm}. It tells us that, as long as we are not at the minimal cost allowed matching, we will be able to perform the augmentation procedure on a path from $s$ to $t$ in $G_{\rm flow}$.
Note that since the augmentation procedure decreases cost by $1$, and since the cost of any allowed matching is $O(n)$, then the previous lemma immediately implies that the minimum cost allowed matching will be found after $O(n)$ rounds of the augmentation procedure. By Proposition \ref{prop:opcount} and Proposition \ref{prop:allexists}, this immediately implies the minimum cost allowed matching can be found in $O(mn)$ operations (provided the graph has no isolated nodes; recall that we did $O(m+n)$ pre-processing to ensure this).

 Our goal in the remainder of the paper is improve this to $O(m \sqrt{n})$.  This is possible because Lemma \ref{lemma:pathscount} tells us that the number of shortest paths from $s$ to $t$ in $G_{\rm flow}(M)$ is an upper bound on the distance to optimality $c(M)-c^*$. If we can argue that the number of paths from $s$ to $t$ decreases fast as a result of the augmentation procedure, improved bounds will be obtained. 

Our strategy, inspired, of course, by the Hopcroft-Karp algorithm, will be to argue that the distance from $s$ to $t$ increases as a result of the augmentation procedure. Since all the paths from $s$ to $t$ are vertex disjoint, lower bounds on the distance from $s$ to $t$ immediately imply upper bounds on the number of shortest paths from $s$ to $t$.

\bigskip

\bigskip

Note that, due to the way of $G_{\rm flow}$ is constructed, performing the augmentation procedure on a collection of paths is the same as reversing the direction of all the core edges in these paths (as then making the necessary modifications to links involving $s,t$ and the $s_i, n^{(i)}_j$ stemming from the changes in which nodes are matched). Consequently, we begin with a sequence of lemmas about the effect of reversing edges on distances in directed graphs.

 Recall that a source node in a directed graph is a graph with no incoming links while a sink node in a directed graph is a graph with no outgoing links. Since we will need to compare distances in different graphs, we adopt the notation $d(a,b,G)$ to denote the distance from node $a$ to node $b$ in the directed graph $G$. Given a directed graph $G$ and nodes $s,t$, we define $S(s,t,G)$ to denote the set of arcs which lie on a shortest paths from $s$ to $t$ in $G$.

\begin{lemma} Consider a directed graph $\Go$, let $S'$ be any subset of $S(s,t,\Go)$ and suppose we reverse the orientation of every arc in $S'$. Call the resulting graph $\Gs$. Then, for any node $q$ we have that $d(s,q, \Go) \leq d(s,q, \Gs)$. \label{lemma:rev:nonincr}
\end{lemma}

\begin{proof} We argue by induction on $d(s,q, \Gs)$. Indeed, nodes $q$ which have $d(s,q, \Gs) = 1$ must have a link from $s$ to $q$ in $\Gs$. This link must be in $\Go$, since no link incoming to $s$ was reversed since no link incoming to $s$ could be on a shortest path starting at $s$. Thus we have $d(s,q, \Go) = 1$.  

Suppose now we have established the proposition for all nodes $q$ with $d(s,q,\Gs) \leq k$. Consider a node $q'$ with $d(s,q',\Gs) = k+1$. Let $p$ be the predecessor of $q'$ on any shortest path from $s$ to $q'$ in $\Gs$. We have $d(s,p, \Gs) = k$ and consequently by the inductive hypothesis, we know that $d(s, p, \Go) \leq k$. We now consider two cases.

In the first case, we have that the link $(p,q')$ is present in $\Go$. In this case, $d(s,q',\Go) \leq k+1$ and we are finished. 

In the second case, we have that the link $(p,q')$ is not present in $\Go$. This means the link $(q',p)$ was present in $\Go$ and was reversed by the procedure in the lemma. Thus $(q',p)$ is on a shortest path starting from $s$ in $\Go$. This means $d(s,q',\Go) \leq d(s,p,\Go) \leq k < k+1$, and we are finished. 
\end{proof}

The previous lemma says that reversing edge directions on a collection of shortest paths cannot decrease distance. The next lemma uses this to strengthen the
conclusion further, namely that this procedure will
strictly increase distances for some nodes.

\begin{lemma}Consider a directed graph $\Go$, let $S'$ be any subset of $S(s,t,\Go)$ and suppose we reverse the orientation of every arc in $S'$. Call the resulting graph $\Gs$.  

Suppose further that a shortest path from $s$ to $q$ in $\Gs$ has the last edge $(p,q)$ which was reversed. Then we have $d(s,q, \Go) < d(s,q, \Gs)$. \label{lemma:rev:firstincr}
\end{lemma}

\begin{proof} The proof is essentially a more careful examination of the proof of Lemma \ref{lemma:rev:nonincr} just given. Note that
\[ d(s,q, \Gs) =  d(s,p,\Gs)+1 \geq d(s,p, \Go)+1, \] where we used Lemma \ref{lemma:rev:nonincr} for the last inequality. Now since  $(q,p)$ was on a shortest path from $s$ to $t$ in $\Go$, we have 
\[ d(s,p,\Go) = d(s,q,\Go)+1 \] so putting the last two inequalities together,
\[ d(s,q, \Gs) \geq d(s,q,\Go) + 2. \]
\end{proof}

As a consequence of these two lemmas, we can conclude that the distance from $s$ to $t$ strictly increases if we reverse enough arcs. This is formally proved in the following lemma. We will say that a set $S' \subset S(s,t,G)$ is {\em saturating} if every shortest path from $s$ to $t$ in $G$ has at least one arc in $S'$. 

\begin{lemma} Consider a directed graph $\Go$, let $S'$ be a saturating subset of $S(s,t,\Go)$ and suppose we reverse the orientation of every arc in $S'$. Call the resulting graph $\Gs$.  Then $d(s,t,\Go) < d(s,t, \Gs)$. \label{lemma:mainincr}
\end{lemma}

\begin{proof} We know from Lemma \ref{lemma:rev:nonincr} that $d(s,t, \Go) \leq d(s,t, \Gs)$. Suppose it was true that $d(s,t, \Go) = d(s,t, \Gs)$.
 Consider a shortest path from $s$ to $t$ in $\Gs$ and observe that by definition this path must have an arc in $S'$. 

Suppose the edge $(a,b)$ is the last arc on this path in $S'$. By Lemma \ref{lemma:rev:firstincr}, we have that $d(s,b,\Go) < d(s,b, \Gs)$. Since the shortest path from $b$ to $t$ does not involve any reversed edges, this proves the lemma. 
\end{proof}

\bigskip

\medskip

We now shift back to the study of the augmentation procedure. As a consequence of the lemma we have just proved, we will show that each round of the augmentation procedure increases the distance from $s$ to $t$ in the graph $G_{\rm flow}$. 

\begin{lemma} Let $M$ be an allowed matching and let $M'$ be the matching obtained by performing one iteration of the augmentation algorithm (i.e., by performing the augmentation procedure along a maximal collection of vertex-disjoint shortest paths from $s$ to $t$ in $G_{\rm flow}(M)$). Then \[ d(s,t, G_{\rm flow}(M)) < d(s,t, G_{\rm flow}(M')). \]  \label{lemma:stincr}
\end{lemma}

\begin{proof} Let $Q$ be the set of edges $(a_{\rm dst}, b_{\rm dst})$ on the maximal vertex-disjoint collection of shortest paths in $G_{\rm flow}(M)$ from $s$ to $t$ on which the augmentation is made. We claim that $G_{\rm flow}(M')$ may be directly constructed from $G_{\rm flow}(M)$ via a sequence of steps, each of which can be described as follows:

\begin{enumerate} \item Add edges from each $b_{\rm dst}$ such that $(a_{\rm dst}, b_{\rm dst})$ belongs to $Q$ for some $a_{\rm dst}$ to every other node in the $[Y_j]_d$ containing $b_{\rm dst}$ which belongs to $F^c$. 
 \item Reverse all edges along a saturating collection of shortest paths from $s$ to $t$.\item Deletes some of the nodes $s_i$ and all their incoming and outgoing edges.
\item Delete the links from $s$ to some of the source nods. 
\item Decrease the number of nodes $n^{(i)}_j$ corresponding to each source strongly connected component remove
some of the edges going to these to account for changes in the set of matched nodes. 
\item Delete all edges $(a_{\rm dst}, c_{\rm dst})$ where $(a_{\rm dst}, b_{\rm dst})$ belongs to $Q$ for some $b_{\rm dst}$. 
\end{enumerate}

Indeed, with the exception of the addition of the adjective {\em saturating} in step (2), this is simply a recitation of all the ways in which $G_{\rm flow}(M')$ differs from $G_{\rm flow}(M)$. Let us describe the changes made in the augmentation procedure and note how they correspond to the above list. 

Consider the edges $(a_{\rm dst}, b_{\rm dst})$ taken in some source connected component $[Y_j]_d$. Then this source connected component still has a single unmatched node after the augmentation procedure. Since the new unmatched node is $b_{\rm dst}$, we need to add edges going from it to every other node in $[Y_j]_d \cap F^c$ - this is done in step (1). We also need to remove edges which are going from $a_{\rm dst}$ to nodes in $[Y_j]_d \cap F^c$ -  this is done in step 6. Moreover, some components which had no unmatched nodes now have unmatched nodes, and the corresponding vertices $s_i$ need to be deleted. This is step (3) 
Furthermore, some source nodes which were unmatched become matched and we need to delete links going to them from $s$ - this is step 4. When a source connected component changes the number of its unmatched nodes, the number of nodes $n^{(i)}_j$ needs to be adjusted - this is step 5.  Note that destination nodes outside of $[X_i]_d$ can only go from unmatched-to-matched as a result of
the augmentation procedure, so the number of $n^{(i)}_j$ can only decrease, and edges going from destination nodes to nodes
$n^{(i)}_j$ can only need to be deleted.

Finally, that this is a complete listing follows from the fact that several things do not happen:
\begin{itemize} \item By Lemma \ref{lemma:donot}, item (1) we do not need to add any edges going from $s$ to any source nodes. 
\item By Lemma \ref{lemma:donot}, item (2) no new nodes $n^{(i)}_j$ need to be introduced.
\item By Lemma \ref{lemma:donot}, item (3) we do not need to add any nodes $s_i$. 
\item Since no destination nodes outside of some $[X_i]_d$ goes from matched to unmatched as a result of the augmentation procedure, we do not need to add edges going from newly unmatched destination nodes to $t$. 
\end{itemize}

To summarize, it is only the correspondence of item (2) to the augmentation procedure which is nontrivial. While it is immediate that augmentation on a collection of paths has the effect of reversing all the core edges on  those paths, two things need to be shown. First, we need to show that the maximal collection of vertex-disjoint shortest paths in the augmentation procedure is {\em still} a maximal collection of vertex-disjoint shortest paths from
$s$ to $t$ once we've added extra edges in step (1). Secondly, we need to argue that reversing the edges in this maximal {\em vertex-disjoint} collections of shortest paths from $s$ to $t$ can be thought of as reversing the edges in a {\em saturating} collection of shortest paths from $s$ to $t$.

Indeed, that the shortest path collection used by the augmentation procedure is still a maximal collection of vertex-disjoint shortest paths after extra edges are added in step (1), simply observe that none of the extra edges can be on a shortest path from $s$ to $t$ since they go from $b_{\rm dst}$ to $c_{\rm dst}$ where $(a_{\rm dst}, c_{\rm dst})$ already existed and $a_{\rm dst}$ had a shorter distance from $s$ than $b_{\rm dst}$ since the edge $(a_{\rm bst}, b_{\rm dst})$ was on a shortest path from $s$ to $t$ in $G_{\rm flow}(M)$.

As for saturating, it is strictly speaking not true that the maximal vertex-disjoint paths reversed by the augmentation procedure will be saturating. Hence we will use the following trick: in step (2) let us reverse more edges than just done by the augmentation procedure. Specifically, if the augmentation procedure reversed the edge $(e_{\rm dst}, n^{(i)}_j)$, then we reverse {\bf all} links going out from $e_{\rm dst}$. Note that $e_{\rm dst}$ is unmatched in this case and all of its outgoing links go to some nodes $n^{(i)}_k$, so that this additional reversal of edges makes no difference: the extra edges we reverse will get deleted in step (5) anyway as we account for $e_{\rm dst}$ becoming matched. We define ${\cal Q}$ to be the set of edges reversed in this way.

We argue that ${\cal Q}$ is indeed saturating. Indeed, suppose there is another shortest path $p'$ from $s$ to $t$ which is edge disjoint from the set of edges ${\cal Q}$. Since ${\cal Q}$ contains the maximal collection of vertex-disjoint paths from $s$ to $t$ found by the augmentation procedure, we have that $p'$ shares a vertex with some arc in ${\cal Q}$. This vertex cannot be some $s_i$ or some $n^{(i)}_j$, since the former has only one incoming edge while the latter has only one outgoing edge, and edge disjointness of $p'$ from ${\cal Q}$ is immediately contradicted. Thus it must be a core edge that $p'$ has in common with some path in ${\cal Q}$. 

Could it be a matched destination node? No, because the only outgoing link from a matched destination node $v_{\rm dst}$ lead to the source node $u_{\rm src}$ such that $(u_{\rm src}, v_{\rm dst})$ is in the twin matching of $M$, so that edge-disjointness from ${\cal Q}$ is immediately contradicted. Could it be an unmatched destination node $u_{\rm dst}$? No, because the only outgoing links from such a node lead only to $t$ or to one or more $n^{(i)}_j$. In the first case, edge-disjointness is immediately contradicted. In the second case, observe that all outgoing links from such a $u_{\rm dst}$ to some $n^{(i)}_j$ are in ${\cal Q}$ by the definition of ${\cal Q}$, so that edge disjointness is again contradicted. Could it be a source vertex? No, because source vertices have exactly one incoming arc in $G_{\rm flow}$ (source vertices without outgoing links in $M$ have an incoming arc from $s$ while source vertices with outgoing links in $M$ have a single incoming arc from a destination vertex). This proves that the action of the augmentation procedure may be thought of as reversing all edges along a saturating collection of arcs.

Having shown that indeed $G_{\rm flow}(M')$ can be constructed from $G_{\rm flow}(M)$ via steps (1)-(6), we now argue that performing these six steps 
increases distance from $s$ to $t$. Indeed, it is immediate that steps (3),(4),(6) cannot decrease the distance from $s$ to $t$. 
Step (5) cannot decrease distance from $s$ to $t$ because by Lemma \ref{lemma:donot}, item (2) no new nodes $n^{(i)}_j$ need to be added for source strongly connected components which did not previously have them. 

Now step (2) increases
the distance from $s$ to $t$ by one as a consequence of Lemma \ref{lemma:stincr}. As for step (1), it cannot decrease the distance from $s$ to $t$ because the added edges cannot lie on any shortest path from $s$ to $t$ due to the fact that since, once again, $(a_{\rm dst}, b_{\rm dst})$ lies on a shortest path from $s$ to $t$ in $G_{\rm flow}(M)$. This concludes the proof.

\end{proof}



\bigskip

\medskip

Having proven the key lemmas - namely Lemma \ref{lemma:stincr} just shown as well as Lemma \ref{lemma:pathscount} proved earlier - it is now 
easy to complete the proof of our main result, Theorem \ref{thm}.

\begin{proof}[Proof of Theorem \ref{thm}] By Lemma \ref{lemma:pathscount}, once we cannot find a shortest path from $s$ to $t$ in $G_{\rm flow}$ to do the augmentation procedure, we have found the minimum cost allowed matching.  By Lemma \ref{lemma:stincr}, after $\sqrt{n}$ iterations of the algorithm, we have that the distance from $s$ to $t$ in $G_{\rm flow}$ is at least $\sqrt{n}$. Since the total number of nodes in $G_{\rm flow}$ is at most $3n+2 \leq 5n$, this means the number of vertex disjoint shortest paths from $s$ to $t$ is at most $5\sqrt{n}$ after $\sqrt{n}$ augmentations. By Lemma \ref{lemma:pathscount}, this means the matching at this point of the algorithm is at most $5\sqrt{n}$ more costly than the optimal, which means the optimal matching is found after $5\sqrt{n}$ more iterations of the augmentation algorithm. Thus the total number of iterations taken by the augmentation algorithm is bounded by $6 \sqrt{n}$.
\end{proof}


\begin{small}

\end{small}
\end{document}